\documentclass[12pt,reqno]{article}

\usepackage[usenames]{color}
\usepackage{amssymb}
\usepackage{graphicx}
\usepackage{amscd}

\usepackage[colorlinks=true,
linkcolor=webgreen,
filecolor=webbrown,
citecolor=webgreen]{hyperref}

\definecolor{webgreen}{rgb}{0,.5,0}
\definecolor{webbrown}{rgb}{.6,0,0}

\usepackage{color}
\usepackage{fullpage}
\usepackage{float}

\usepackage{graphics,amsmath,amssymb}
\usepackage{amsthm}
\usepackage{amsfonts}
\usepackage{latexsym}
\usepackage{epsf}

\begin{document}

\begin{center}
\end{center}

\theoremstyle{plain}
\newtheorem{theorem}{Theorem}
\newtheorem{corollary}[theorem]{Corollary}
\newtheorem{lemma}[theorem]{Lemma}
\newtheorem{proposition}[theorem]{Proposition}

\theoremstyle{definition}
\newtheorem{definition}[theorem]{Definition}
\newtheorem{example}[theorem]{Example}
\newtheorem{conjecture}[theorem]{Conjecture}

\theoremstyle{remark}
\newtheorem{remark}[theorem]{Remark}

\newcommand{\Q}{{\mathbb Q}}

\def\modd#1 #2{#1\ \mbox{\rm (mod}\ #2\mbox{\rm )}}

\begin{center}
\vskip 1cm{\LARGE\bf Explicit Asymptotics for Signed Binomial Sums and Applications to Carnevale-Voll Conjecture}
\vskip 1cm
\large
Laurent Habsieger \\
Universit\'e de Lyon, CNRS UMR 5208 \\
Universit\' e Claude Bernard Lyon 1, Institut Camille Jordan\\
43 boulevard du 11 novembre 1918 \\
69622 Villeurbanne Cedex
France \\
\href{mailto:habsieger@math.univ-lyon1.fr}{\tt  habsieger@math.univ-lyon1.fr}\\
\end{center}

\vskip .2 in

\begin{abstract}
Carnevale and Voll conjectured that $\sum_j (-1)^j{\lambda_1\choose j}{\lambda_2\choose j}\neq0$ when $\lambda_1$ and $\lambda_2$ are two distinct integers. 
We check the conjecture when either $\lambda_2$ or $\lambda_1-\lambda_2$ is small.
We investigate the asymptotic behaviour of their sum when the ratio $r:=\lambda_1/\lambda_2$ is fixed and $\lambda_2$ goes to infinity.
We find an explicit range $r\ge 5.8362$ on which the conjecture is true.
We show that the conjecture is almost surely true for any fixed $r$. 
For $r$ close to $1$, we give several explicit intervals on which the conjecture is also true.
\end{abstract}

\section{Introduction}

Carnevale and Voll \cite{C-V}  studied Dirichlet series enumerating orbits of Cartesian products of maps whose orbits distributions are modelled on the distributions of finite index subgroups of free abelian groups of finite ranks. For Cartesian products of more than three maps they establish a natural boundary for meromorphic continuation. For products of two maps, they formulate two combinatorial conjectures that prove the existence of such a natural boundary. These conjectures state that some explicit polynomials have no unitary factors, i.e., polynomial factors that,
for a suitable choice of variables, are univariate and have all their zeroes on the unit circle. The polynomials related to their Conjecture A \cite{C-V} are given by:
$$C_{\lambda_1,\lambda_2}(x,1)=\sum_{j=0}^{\lambda_2} {\lambda_1\choose j}{\lambda_2\choose j} x^j$$
for positive integers $\lambda_1\ge \lambda_2$, and the conjectured property is the following.

\begin{conjecture}\label{conj}
 Let $\lambda_1> \lambda_2$ be two positive integers. Then $C_{\lambda_1,\lambda_2}(-1,1)\neq0$.
\end{conjecture}

Note that
$$\aligned
C_{\lambda,\lambda}(-1,1)&=\sum_{j=0}^{\lambda} (-1)^j {\lambda\choose j}{\lambda\choose \lambda-j}  
= \oint (1-z)^{\lambda}(1+z)^{\lambda} \frac{dz}{2i\pi z^{\lambda+1}}
=\oint (1-z^2)^{\lambda} \frac{dz}{2i\pi z^{\lambda+1}}\\
&=\begin{cases} (-1)^{\lambda/2} {\lambda\choose \lambda/2} &\text{if $\lambda$ is even;} \\
0 &\text{if $\lambda$ is odd.} \end{cases}
\endaligned$$
This explains why the case $\lambda_1=\lambda_2$ is excluded.

Carnevale and Voll \cite{C-V} reported  that Stanton pointed out the following property: 
the alternating summands have increasing absolute values for $\lambda_1>\lambda_2(\lambda_2+1)-1$, which
shows the next result.

\begin{proposition} \label{thmstanton}
For all $\lambda_2$ and $\lambda_1>\lambda_2(\lambda_2+1)-1$, we have $C_{\lambda_1,\lambda_2}(-1,1)\neq0$.
\end{proposition}

For a fixed $\lambda_2$, the sum $C_{\lambda_1,\lambda_2}(-1,1)$ is a polynomial in $\lambda_1$ of degree $\lambda_2$. The first values are 
$C_{\lambda_1,0}(-1,1)=1$, $C_{\lambda_1,1}(-1,1)=1-\lambda_1$. Moreover, for $2\le\lambda_2\le240$, we checked with {\tt Maple}  that it is an irreducible polynomial over $\Q$ when $\lambda_2$ is even, and that it is the product of $\lambda_1-\lambda_2$ by an irreducible polynomial over $\Q$ when $\lambda_2$ is odd. 
The even case required much less time (238 seconds) than the odd case (63908 seconds). Since an irreducible polynomial of degree at least $2$ cannot have an integer zero, we deduce that the conjecture is true for the first values of $\lambda_2$.

\begin{proposition} \label{firstvalues}
For all $1\le \lambda_2\le 240$ and $\lambda_1>\lambda_2$, we have $C_{\lambda_1,\lambda_2}(-1,1)\neq0$.
\end{proposition}

The aim of this paper is to study the size of $C_{\lambda_1,\lambda_2}(-1,1)$ when $\lambda_1$ and $\lambda_2$ are large enough. We shall give explicit estimates in order to
extend the range of validity of Conjecture \ref{conj}.

We start by relating $C_{\lambda_1,\lambda_2}(x,1)$ to a complex integral formula:
$$C_{\lambda_1,\lambda_2}(x,1)= \oint (1+z)^{\lambda_1}(z+x)^{\lambda_2} \frac{dz}{2i\pi z^{\lambda_2+1}}\,,$$
where the path is a simple one around $0$. This will always be the case from now on. We thus have
$$(-1)^{\lambda_2}C_{\lambda_1,\lambda_2}(-1,1)= \oint (1+z)^{\lambda_1}(1-z)^{\lambda_2} \frac{dz}{2i\pi z^{\lambda_2+1}}\,.$$

Put $\lambda=\lambda_2$ and $r=\lambda_1/\lambda_2>1$. We get
\begin{equation}\label{C}
(-1)^{\lambda_2}C_{\lambda_1,\lambda_2}(-1,1)= \oint \exp(\lambda f(z)) \frac{dz}{2i\pi z}\,,
\end{equation}
with 
\begin{equation}\label{deff}
f(z)=f(r,z):=r\log(1+z)+\log(1-z)-\log z\,.
\end{equation}
We need to find the right path to be able to find the asymptotic behaviour of this kind of integral when $\lambda$ goes to infinity.
Let us take $z=\rho e^{i\theta}$ with $-\pi\le \theta\le \pi$. The parameter $\rho$ will be optimal when $f'$ vanishes on the path. Since
\begin{equation}\label{deff'}
f'(z)=\frac{r}{1+z}-\frac{1}{1-z}-\frac{1}{z}=\frac{-rz^2+(r-1)z-1}{z(1-z^2)}\,,
\end{equation}
and $(r-1)^2-4r=r^2-6r+1=(r-3-2\sqrt2)(r-3+2\sqrt2)$, we need to distinguish several cases: $1<r<3+2\sqrt2$, $r=3+2\sqrt2$ and $r>3+2\sqrt2$.
We thus define
\begin{equation}\label{defrho}
\rho= \rho(r):=
\begin{cases} 
\frac{r-1-\sqrt{r^2-6r+1}}{2r} &\text{if $r\ge3+2\sqrt2$;}\\
\frac{1}{\sqrt r}=\left\vert \frac{r-1\pm i\sqrt{-r^2+6r-1}}{2r}\right\vert&\text{if $r\le 3+2\sqrt2$.}
\end{cases}
\end{equation}
By \eqref{C} we want to study the integral
\begin{equation}\label{defI}
I(\lambda)=I(r,\lambda):=\int_{-\pi}^{\pi} \exp(\lambda f(\rho e^{i\theta})) \frac{d\theta}{2\pi}\,.
\end{equation}

In the case $r>3+2\sqrt2$, we find $f(\rho e^{i\theta})=f(\rho)-M\frac{\theta^2}{2}+o(\theta^3)$ for some positive real number $M$, when $\theta$ goes to $0$.
We therefore  find $I(\lambda)\sim \int_{-\infty}^{\infty} \exp\left(-\lambda M\frac{\theta^2}{2}\right)\frac{d\theta}{2\pi}$. We prove an effective
version of this equivalence.

\begin{theorem} \label{thas+}
For $r>3+2\sqrt2$, put $M=M(r):=\rho^2f''(\rho)$. Then $M>0$ and we have
$$\left\vert  \frac{\sqrt{2\pi \lambda M}}{\exp(\lambda f(\rho))} I(\lambda)-1\right\vert
\le \frac{3(3+2\sqrt2)\pi^5}{256\lambda M^2}+\frac{5(3+2\sqrt2)}{24\lambda M^3}+\frac{\sqrt2\exp\left(-\lambda M\frac{\pi^2}{2}\right) }{\pi^{3/2}\sqrt{\lambda M}}\,.$$
\end{theorem}

This theorem shows that, for any $r$, we have $I(\lambda)\neq0$ for $\lambda$ large enough. We use Proposition \ref{firstvalues} and further tools to deduce a large range of validity for Conjecture \ref{conj}.

\begin{theorem} \label{th0+} For  $\lambda_1\ge 5.8362\lambda_2>0$, we have $C_{\lambda_1,\lambda_2}(-1,1)\neq0$.
\end{theorem}

Note that the value $5.8362$ is close to $3+2\sqrt2=5.82842\dots$, the limit of the method. 

Since $r=\lambda_1/\lambda_2$ is a rational number, the case $r=3+2\sqrt2$ cannot occur, so we shall not detail it. The same method would provide an effective version of the equality
$$\frac{2^{2/3}3^{1/6}\pi(\sqrt2+1)^{1/3}}{\Gamma(1/3)}\times
\frac{\lambda^{1/3}}{2^{(2+\sqrt2)\lambda}} I(3+2\sqrt2,\lambda)=1+O\left(\frac{1}{\lambda^{1/6}}\right)\,.$$

In the case $1<r<3+2\sqrt2$, the situation is quite different. There are two conjugate points on the integrating circle where $f'$ vanishes. Their contributions partially cancel each other, so we cannot get an exact equivalent term:
for some choices of $(\lambda_1,\lambda_2)$, the implied constant may be really small. The analog of Theorem \ref{thas+} has indeed the following form.

\begin{theorem} \label{thas-} For $1<r<3+2\sqrt2$, define  $\gamma_1=\arccos\frac{3r-1}{2\sqrt2r}$, $\gamma_2=-\arccos\frac{r-3}{2\sqrt2}$ and 
$\gamma_3=\frac{1}{2}\arcsin\frac{(r-1)^2}{4r}$.  For $\lambda\ge \frac{512r^{3/2}}{(r+1)(-r^2+6r-1)^{3/2}}$, we have
$$\left\vert \frac{(-r^2+6r-1)^{1/4} \sqrt{\pi\lambda}}{2^{1+\frac{(r+1)\lambda}{2} }} I(\lambda)
- \cos\left((r\gamma_1+\gamma_2)\lambda+\gamma_3\right) \right\vert
\le \frac{16336}{\sqrt\lambda (-r^2+6r-1)^{11/4}} \,.$$
\end{theorem}

It seems quite difficult to find a lower bound for $\vert \cos\left(\lambda_1\gamma_1+\lambda_2\gamma_2+\gamma_3\right)\vert$ for every $\lambda_1$, $\lambda_2$,
and thus to show that $I(\lambda)\neq0$. However, we can upper bound the number of possible exceptions.

\begin{theorem}\label{thexceptions} For $r>1$, we have
$$\displaylines{
\#\{ (\lambda_1,\lambda_2)\ :\ \lambda_1=r\lambda_2\,,\ C_{\lambda_1,\lambda_2}(-1,1)\neq0\,,\ \lambda_2\le x\}
\hfill\cr\hfill
\le\begin{cases}
O_r(1) &\text{if $r>3+2\sqrt2$;}\\
\frac{102644}{ (-r^2+6r-1)^{11/4}\log\left(\frac{1+\sqrt5}{2}\right)}\sqrt x\log x+O_r(\sqrt x) &\text{if $1<r<3+2\sqrt2$.}
\end{cases}}$$
\end{theorem}

Note that this theorem implies that, for any fixed $r>1$, we have $C_{\lambda_1,\lambda_2}(-1,1)\neq0$ almost surely.
We can get more explicit estimates when $r$ is close to $1$, following a suggestion of Dennis Stanton. 
This case is also of special interest since $C_{\lambda_1,\lambda_2}(-1,1)$ can be reduced to a sum with at most $\left\lfloor\frac{\lambda_1}{2}\right\rfloor-\left\lceil\frac{\lambda_2}{2}\right\rceil$ terms, using a hypergeometric transformation. This enables us to prove the following analogue of Proposition \ref{firstvalues}.

\begin{proposition} \label{propr1}
For all $1\le \lambda_1-\lambda_2\le 701$, we have $C_{\lambda_1,\lambda_2}(-1,1)\neq0$.
\end{proposition}

We then prove the following theorems.

\begin{theorem} \label{thas1} For $1<\lambda_1/\lambda_2<3+2\sqrt2$, define  $\gamma_1=\arccos\frac{3\lambda_1-\lambda_2}{2\sqrt2\lambda_1}$ and $\gamma_2=-\arccos\frac{\lambda_1-3\lambda_2}{2\sqrt2\lambda_2}$.  
Assume $\lambda_1-\lambda_2\ge702$. We then have
$\left\vert \frac{ \sqrt{\pi\lambda_2}}{2^{\frac{\lambda_1+\lambda_2+1}{2}}} I(\lambda_2)
- \cos\left(\lambda_1\gamma_1+\lambda_2\gamma_2\right) \right\vert
\le 0.0165$ for $\lambda_1-\lambda_2\le  \sqrt{8\pi\lambda_2}$, and
$$\sqrt\lambda_2\left\vert \frac{ \sqrt{\pi\lambda_2}}{2^{\frac{\lambda_1+\lambda_2+1}{2}}} I(\lambda_2)
- \cos\left(\lambda_1\gamma_1+\lambda_2\gamma_2\right) \right\vert
\le\begin{cases}
 1.05882 &\text{if $\log\lambda_2 \le \lambda_1-\lambda_2\le \sqrt{\pi\lambda_2}$;}\\
 1.30775 &\text{if $\sqrt{\pi\lambda_2}\le \lambda_1-\lambda_2\le  \sqrt{2\pi\lambda_2}$;}\\
 1.50929 &\text{if $\sqrt{2\pi\lambda_2}\le \lambda_1-\lambda_2\le  \sqrt{3\pi\lambda_2}$;}\\
 1.68876 &\text{if $\sqrt{3\pi\lambda_2}\le  \lambda_1-\lambda_2\le  \sqrt{4\pi\lambda_2}$;}\\
 1.85482 &\text{if $\sqrt{4\pi\lambda_2}\le \lambda_1-\lambda_2\le  \sqrt{5\pi\lambda_2}$;}\\
 2.01189 &\text{if $\sqrt{5\pi\lambda_2}\le \lambda_1-\lambda_2\le  \sqrt{6\pi\lambda_2}$;}\\
 2.1626   &\text{if $ \sqrt{6\pi\lambda_2}\le \lambda_1-\lambda_2\le  \sqrt{7\pi\lambda_2}$;}\\
 2.30865 &\text{if $\sqrt{7\pi\lambda_2}\le \lambda_1-\lambda_2\le  \sqrt{8\pi\lambda_2}$.}\\
 \end{cases}$$
\end{theorem}

\begin{theorem} \label{thr1} We have $C_{\lambda_1,\lambda_2}(-1,1)\neq0$ in the following cases:
\begin{itemize}
\item $\lambda_1+\lambda_2\equiv \modd{0} {4}$: $\lambda_2<\lambda_1\le\lambda_2+\sqrt{2\pi\lambda_2}-1.0443$ or $\lambda_2+\sqrt{2\pi\lambda_2}+3.1407\le\lambda_1\le \lambda_2+\sqrt{6\pi\lambda_2}-0.9275$;
\item $\lambda_1+\lambda_2\equiv \modd{1} {4}$: $\lambda_2<\lambda_1\le\lambda_2+\sqrt{3\pi\lambda_2}-0.984$ or $\lambda_2+\sqrt{3\pi\lambda_2}+3.8433\le \lambda_1\le \sqrt{7\pi\lambda_2}-0.9231$;
\item $\lambda_1+\lambda_2\equiv \modd{2} {4}$, $\lambda_2+\max\left(702,2.0582\lambda_2^{1/4}\right)\le\lambda_1\le\lambda_2+\sqrt{2\pi\lambda_2}-0.9535$ or 
$2\sqrt{\pi\lambda_2}+4.5938\le \lambda_1-\lambda_2\le 2\sqrt{2\pi\lambda_2}-0.9218$;
\item $\lambda_1+\lambda_2\equiv \modd{3} {4}$: $\lambda_2<\lambda_1\le\lambda_2+\sqrt{\pi\lambda_2}-1.1958$ or $\lambda_2+\sqrt{\pi\lambda_2}+2.5913\le\lambda_1\le \lambda_2+\sqrt{5\pi\lambda_2}-0.9367$.
\end{itemize}
\end{theorem}

We chose to study in Theorem \ref{thr1} what happens before and after the first gap. The results obtained show that we miss at most six values in the first gap, which is quite small.

In the next section we study the case $r>3+2\sqrt2$. We investigate the case $1<r<3+2\sqrt2$ in Section 3, and focus on the case $r$ close to $1$ in Section 4. We end this paper with some remarks and conjectures.

Before starting our studies, let us note that
\begin{equation}\label{deff''}
\aligned
f''(z) &= -\frac{r}{(1+z)^2}-\frac{1}{(1-z)^2}+\frac{1}{z^2}\,,\\
f'''(z) &= \frac{2r}{(1+z)^3}-\frac{2}{(1-z)^3}-\frac{2}{z^3}\,,
\endaligned
\end{equation}
and let us define the even function
$g(\theta)=g(r,\theta):= \Re(f(\rho e^{i\theta}))$ and the odd function
$h(\theta)=h(r,\theta):= \Im(f(\rho e^{i\theta}))$.

Since $\overline{ f(\rho e^{i\theta})}= f(\rho e^{-i\theta})$, we have the useful expressions
\begin{equation}\label{equa1}
\int_{-\pi}^{\pi} \exp(\lambda f(\rho e^{i\theta})) \frac{d\theta}{2\pi}
= \int_{-\pi}^{\pi} \exp(\lambda g(\theta)) \cos\left(\lambda h(\theta)\right)\frac{d\theta}{2\pi}
= \int_0^{\pi} \exp(\lambda g(\theta)) \cos\left(\lambda h(\theta)\right)\frac{d\theta}{\pi}\,.
\end{equation}

\section{The case $r> 3+2\sqrt2$}

\subsection{General properties}

A straightforward calculation shows that
\begin{equation}\label{signrho'}
\rho'(r) = \frac{\sqrt{1-6r+r^2}+1-3r}{2r^2\sqrt{1-6r+r^2}}<0\,,
\end{equation}
and we get in this case
\begin{equation}\label{estrho+}
0<\rho< \rho(3+2\sqrt2)=\sqrt2-1\,.
\end{equation}
From \eqref{deff'} we have 
\begin{equation}\label{estr}
r=\frac{1+\rho}{\rho(1-\rho)}
\end{equation}
and we deduce from \eqref{deff''} and \eqref{estrho+}
\begin{equation}\label{defM}
M =M(r) :=\rho^2f''(\rho)=\frac{r\rho}{(1+\rho)^2}-\frac{\rho}{(1-\rho)^2}
=\frac{1-2\rho-\rho^2}{(1+\rho)(1-\rho)^2}>0\,.
\end{equation}
Another straightforward calculation shows that
\begin{equation}\label{signM'}
M'(r)=-\rho'(r) \frac{1+\rho+5\rho^2+\rho^3}{(1+\rho)^2(1-\rho)^3}>0\,,
\end{equation}
and we get
\begin{equation}\label{estM+}
0<M<1\,.
\end{equation}

Let us now state a key lemma, which shows what kind of estimates are needed and how to use them to get results for $I(\lambda)$.

\begin{lemma}\label{key+} Let $\delta\in\lbrack0,\pi\rbrack$. Assume that
\begin{enumerate}
\item $g(\theta)-g(0)\le -KM\frac{\theta^2}{2}$, for some constant $0<K\le1$,
\item $\left\vert g(\theta)-g(0) +M\frac{\theta^2}{2}\right\vert \le C_g \theta^4$, for some positive constant $C_g$,
\item $\vert h(\theta)\vert\le C_h\vert\theta\vert^3$,  for some positive constant $C_h$,
\end{enumerate}
for $0\le\vert\theta\vert\le\delta$. We then have
$$ \displaylines{\left\vert\frac{\sqrt{2\pi\lambda M}}{\exp\left(\lambda f(\rho)\right)} I(\lambda)-1\right\vert
\hfill\cr\hfill
<\frac{ 3C_g}{\lambda K^{5/2}M^2 }+
\frac{15 C_h^2}{2\lambda M^3}+\frac{\sqrt2 }{\delta\sqrt{\pi\lambda M}}\exp\left(-\lambda M\frac{\delta^2}{2}\right)
+\left(1-\frac{\delta}{\pi}\right)\sqrt{2\pi\lambda M} \exp\left(-KM\frac{\delta^2}{2}\right)\,.}$$
\end{lemma}

\begin{proof} Note that
\begin{equation}\label{expcos}
\left\vert e^x\cos y-1\right\vert\le \left\vert \left(e^x-1\right)\cos y\right\vert +1-\cos y\le  \vert x\vert e^{\max(x,0)}+\frac{y^2}{2}
\end{equation}
for any real numbers $x$ and $y$. 

Let us use this property with $x=\lambda \left(g(\theta)-g(0)+M\frac{\theta^2}{2}\right)$ and $y=\lambda h(\theta)$. 
Condition 1 provides the upper bound $\max(x,0)\le \lambda (1-K)M\frac{\theta^2}{2}$.
Conditions 2 and 3  imply $\vert x\vert\le \lambda C_g\theta^4$ and  $\vert y\vert\le \lambda C_h\vert \theta\vert^3$, respectively.
We now deduce  from \eqref{expcos} and these bounds:
$$\displaylines{ \exp\left(-\lambda f(\rho)\right)
  \left\vert\int_{-\delta}^{\delta} \exp(\lambda f(\rho e^{i\theta})) \frac{d\theta}{2\pi} -\int_{-\delta}^{\delta} \exp\left(\lambda\left(f(\rho)-M\frac{\theta^2}{2}\right) \right)
 \frac{d\theta}{2\pi}\right\vert
\hfill\cr\hfill
\aligned
&=    \left\vert \int_{-\delta}^{\delta} \left(  e^x\cos y-1\right) \exp\left(-\lambda M\frac{\theta^2}{2}\right)  \frac{d\theta}{2\pi}\right\vert\\
&\le \int_{-\delta}^{\delta}\lambda C_g\theta^4  \exp\left(-\lambda KM\frac{\theta^2}{2} \right)  \frac{d\theta}{2\pi}
+\int_{-\delta}^{\delta} \frac{\lambda^2C_h^2\theta^6}{2}\exp\left(-\lambda M\frac{\theta^2}{2} \right)  \frac{d\theta}{2\pi}\\
&< \frac{\lambda C_g\Gamma(5/2)}{2\pi(\lambda KM/2)^{5/2}}+ \frac{\lambda^2C_h^2\Gamma(7/2)}{4\pi(\lambda M/2)^{7/2}}
=\frac{3C_g}{\sqrt{2\pi}\lambda^{3/2}(KM)^{5/2}}+\frac{15C_h^2}{2\sqrt{2\pi}\lambda^{3/2}M^{7/2}}\,.
\endaligned}$$
Since
$$\int_{\delta\le \vert\theta\vert\le+\infty} \exp\left(-\lambda M\frac{\theta^2}{2}\right) \frac{d\theta}{2\pi}
\le \int_{\delta}^{+\infty} \frac{\theta}{\delta} \exp\left(-\lambda M\frac{\theta^2}{2}\right) \frac{d\theta}{\pi}
 =\frac{ \exp\left(-\lambda M\frac{\delta^2}{2}\right) }{\delta\lambda M\pi}$$
and 
$$ \int_{-\infty}^{+\infty} \exp\left(-\lambda M\frac{\theta^2}{2}\right) \frac{d\theta}{2\pi}=\frac{1}{\sqrt{2\pi\lambda M}}\,,$$
we obtain
$$ \left\vert\frac{\sqrt{2\pi\lambda M}}{\exp\left(\lambda f(\rho)\right)}
\int_{-\delta}^{\delta} \exp(\lambda f(\rho e^{i\theta})) \frac{d\theta}{2\pi}-1\right\vert
<\frac{ 3C_g}{\lambda K^{5/2}M^2 }+
\frac{15 C_h^2}{2\lambda M^3}+
\frac{\sqrt2 }{\delta\sqrt{\pi\lambda M}}\exp\left(-\lambda M\frac{\delta^2}{2}\right)\,.$$

To deal with the remaining integral $\int_{\delta\le\vert\theta\vert\le\pi}\exp(\lambda f(\rho e^{i\theta})) \frac{d\theta}{2\pi}$, we first show that $g$ is increasing on $\lbrack-\pi,0\rbrack$ and decreasing on $\lbrack0,\pi\rbrack$.
From the definition 
$$g(\theta)=\Re\left(  f\left(\rho e^{i\theta}\right) \right) =r\log\left\vert1+\rho e^{i\theta}\right\vert+\log\left\vert1-\rho e^{i\theta}\right\vert-\log\rho$$
and \eqref{estr} we deduce 
$$g(\theta)=\frac{1+\rho}{2\rho(1-\rho)}\log\left(1+2\cos\theta \rho+\rho^2\right)+\frac{1}{2}\log\left(1-2\cos\theta \rho+\rho^2\right)-\log\rho$$
and
$$\begin{aligned}
g'(\theta)&=-\frac{1+\rho}{2\rho(1-\rho)}\times\frac{2\rho\sin\theta}{1+2\rho\cos\theta+\rho^2}+\frac{1}{2}\times\frac{2\rho\sin\theta}{1-2\rho\cos\theta+\rho^2}\\
&=-\frac{\sin\theta( 2\rho(1+2\rho-\rho^2)(1-\cos\theta)+(1-\rho^2)(1-2\rho-\rho^2))}{(1-\rho)((1+\rho^2)^2-4\rho^2\cos^2\theta)}\le0\,,
\end{aligned}$$
which proves this intermediate result. We therefore have
$$\displaylines{
\frac{1}{\exp\left(\lambda f(\rho)\right)}
\left\vert \int_{\delta\le\vert\theta\vert\le\pi} \exp(\lambda f(\rho e^{i\theta})) \frac{d\theta}{2\pi}\right\vert
\le \int_{\delta\le\vert\theta\vert\le\pi} \exp\left(\lambda(g(\theta)-g(0))\right) \frac{d\theta}{2\pi}
\hfill\cr\hfill
\le  \frac{\pi-\delta}{2\pi}\left(\exp\left(\lambda(g(\delta)-g(0))\right)+\exp\left(\lambda(g(-\delta)-g(0))\right)\right)
\le  \frac{\pi-\delta}{\pi}\exp\left(-KM\frac{\delta^2}{2}\right)\,,}$$
and the lemma follows.
\end{proof}

\subsection{Estimates for $g$ and $h$}

In this subsection, we obtain explicit versions of Conditions 1-3 in Lemma \ref{key+}. We shall present two kinds of inequalities: a general inequality valid for any $r>3+2\sqrt2$, and a more precise one only valid when $r$ is close to $3+2\sqrt2$. The first ones will provide applications when $r$ is large enough, and the second ones when $r$ is close to $3+2\sqrt2$.

\begin{lemma}\label{r+g} For $r>3+2\sqrt2$ and $\vert\theta\vert\le\pi$, we have
$$g(\theta)-g(0)\le  -\frac{2}{\pi^2} M\theta^2\,.$$
For $3+2\sqrt2<r\le7.686899$ and $\vert\theta\vert\le\pi/3$, we have
$$g(\theta)-g(0)<-M\frac{\theta^2}{2}\,.$$
\end{lemma}

\begin{proof} From the definition of $g$, we get
\begin{equation}\label{exprg}
\aligned
g(\theta)-g(0)&=\frac{r}{2}\log\left(\frac{1+2\rho\cos\theta+\rho^2}{(1+\rho)^2}\right)+\frac{1}{2}\log\left(\frac{1-2\rho\cos\theta+\rho^2}{(1-\rho)^2}\right)\\
&=\frac{r}{2}\log\left(1-\frac{2\rho(1-\cos\theta)}{(1+\rho)^2}\right)+\frac{1}{2}\log\left(1+\frac{2\rho(1-\cos\theta)}{(1-\rho)^2}\right)\,.
\endaligned
\end{equation}
Without loss of generality we may assume $\theta>0$. 
Since  $\log(1+x)\le x$ for $x>-1$, we deduce from \eqref{exprg} the upper bound:
\begin{equation}\label{boundg}
g(\theta)-g(0)\le\left( \frac{\rho}{(1-\rho)^2}-\frac{r\rho}{(1+\rho)^2}\right) (1-\cos\theta)=-2M\sin^2\left(\frac{\theta}{2}\right)
\le -2M\left(\frac{\theta}{\pi}\right)^2\,,
\end{equation}
which proves the first part of the lemma.

From  
 $\log(1+x)\le x-\frac{x^2}{2}+\frac{x^3}{3}$ for $x>-1$, we deduce from\eqref{defM} and  \eqref{exprg}
 \begin{equation}\label{supg}
\aligned
 g(\theta)-g(0)&\le -M(1-\cos\theta) -\left( \frac{\rho^2}{(1-\rho)^4}+\frac{r\rho^2}{(1+\rho)^4}\right) (1-\cos\theta)^2\\
&+ \left( \frac{\rho^3}{(1-\rho)^6}-\frac{r\rho^3}{(1+\rho)^6}\right) \frac{4(1-\cos\theta)^3}{3}\,.\endaligned 
\end{equation}
Let us put 
$$\varphi_1(\theta):=\cos\theta-1+\frac{\theta^2}{2}-\frac{(1-\cos\theta)^2}{6}-\frac{2(1-\cos\theta)^3}{45}\,,$$
so that $\varphi_1(0)=\varphi'_1(0)=0$ and $\varphi''_1(\theta)=\frac{2(1-\cos\theta)^3}{5}$. By Taylor's formula and using the parity of $\varphi_1$, there exists $t_{\theta}\in\lbrack0,\vert\theta\vert\rbrack$
such that
\begin{equation}\label{phi1}
0\le \varphi_1(\theta) = \frac{2(1-\cos t_{\theta})^3}{5}\times\frac{\theta^2}{2}\le \frac{\theta^2(1-\cos\theta)^3}{5}\,.
\end{equation}
From \eqref{supg}, we get
$$\aligned
g(\theta)-g(0)&\le M\left( -\frac{\theta^2}{2}+\frac{(1-\cos\theta)^2}{6}+\frac{2(1-\cos\theta)^3}{45}+\frac{\theta^2(1-\cos\theta)^3}{5}\right)\\
&-\left( \frac{\rho^2}{(1-\rho)^4}+\frac{r\rho^2}{(1+\rho)^4}\right) (1-\cos\theta)^2+ \left( \frac{\rho^3}{(1-\rho)^6}-\frac{r\rho^3}{(1+\rho)^6}\right) \frac{4(1-\cos\theta)^3}{3}\\
&= -M\frac{\theta^2}{2} +c_1(\rho) (1-\cos\theta)^2+c_2(\rho,\theta)(1-\cos\theta)^3\,.
\endaligned$$
with
$$\aligned
c_1(\rho) :&=\frac{M}{6}-\left( \frac{\rho^2}{(1-\rho)^4}+\frac{r\rho^2}{(1+\rho)^4}\right)
=\frac{1-8\rho+9\rho^2-32\rho^3-9\rho^4-8\rho^5-\rho^6}{6(1+\rho)^3(1-\rho)^4}\\
c_2(\rho,\theta):&=M\left(\frac{2}{45}+\frac{\theta^2}{5}\right)+\frac{4}{3}\left( \frac{\rho^3}{(1-\rho)^6}-\frac{r\rho^3}{(1+\rho)^6}\right)
\,.\endaligned$$
For $3+2\sqrt2 <r \le 7.686899$, we have $ 0.19186222\le\rho(r)< \sqrt2-1$, and we check that 
$$c_1(\rho)+(1-\cos\theta)c_2(\rho,\theta)\le c_1(\rho)+\frac{c_2(\rho,\pi/3)}{2}<0$$
for $\vert \theta\vert\le \pi/3$. The lemma follows.
\end{proof}

\begin{lemma}\label{r+g1} Let $c\in\lbrack -1,1\rbrack$. For $r>3+2\sqrt2$ and  $\cos\theta\ge c$, we have 
$$\left\vert g(\theta)-g(0)+M\frac{\theta^2}{2}\right\vert \le  C_g\theta^4\,,$$
with
$$C_g:=\max\left(\frac{1-2\rho-\rho^2}{24(1+\rho)(1-\rho)^2},\frac{\rho^4+6\rho^3+2\rho^2+4\rho-1}{24(1+\rho)(1-\rho)^4} 
+\frac{\rho}{4(1-\rho^2)(1+\rho^2+2\rho c)}\right)\,.$$
Moreover, for $r\le 7.494$, we have
$$C_g=\frac{\rho^4+6\rho^3+2\rho^2+4\rho-1}{24(1+\rho)(1-\rho)^4} +\frac{\rho}{4(1-\rho^2)(1+\rho^2+2\rho c)}\,.$$
\end{lemma}

\begin{proof} It follows from \eqref{boundg} that $g(\theta)-g(0)\le M(\cos\theta-1)$.
The upper bound 
\begin{equation}\label{upg}
g(\theta)-g(0)+M\frac{\theta^2}{2}\le M\frac{\theta^4}{24}=\frac{1-2\rho-\rho^2}{24(1+\rho)(1-\rho)^2}\theta^4
\end{equation}
follows from \eqref{defM} and from $\cos\theta\le 1-\frac{\theta^2}{2}+\frac{\theta^4}{24}$.

Since $\log(1-x)\ge -x-\frac{x^2}{2(1-x)}$ and $\log(1+x)\ge x-\frac{x^2}{2}$ for $x\in(0,1)$, we deduce from \eqref{exprg} the lower bound
$$g(\theta)-g(0)\ge  -M(1-\cos\theta)
-\left( \frac{r\rho^2}{(1+\rho)^4(1-\frac{2\rho(1-\cos\theta)}{(1+\rho)^2})}+\frac{\rho^2}{(1-\rho)^4}\right)(1-\cos\theta)^2\,.$$
From \eqref{phi1} we get
$$\cos\theta-1\ge -\frac{\theta^2}{2}+\frac{(1-\cos\theta)^2}{6}+ \frac{2(1-\cos\theta)^3}{45}\ge-\frac{\theta^2}{2}+\frac{(1-\cos\theta)^2}{6}\,,$$
and we deduce
$g(\theta)-g(0)+M\frac{\theta^2}{2} \ge -\varphi_2(\rho,c)(1-\cos\theta)^2$,
with 
$$\begin{aligned}
\varphi_2(\rho,c):&= -\frac{M}{6}+\frac{r\rho^2}{(1+\rho)^4\left(1-\frac{2\rho(1-c)}{(1+\rho)^2}\right)}+\frac{\rho^2}{(1-\rho)^4}\\
&=\frac{\rho^4+6\rho^3+2\rho^2+4\rho-1}{6(1+\rho)(1-\rho)^4} 
+\frac{\rho}{(1-\rho^2)(1+\rho^2+2\rho c)}\,.
\end{aligned}$$
From $(1-\cos\theta)^2\le\theta^4/4$, we obtain $g(\theta)-g(0)+M\frac{\theta^2}{2} \ge -\max\left(\varphi_2(\rho,c),0\right)\frac{\theta^4}{4}$ and therefore
$$\left\vert g(\theta)-g(0)+M\frac{\theta^2}{2}\right\vert \le \max\left(\frac{\varphi_2(\rho,c)}{4},0,\frac{1-2\rho-\rho^2}{24(1+\rho)(1-\rho)^2}\right)\theta^4= C_g\theta^4\,.$$

We note that
$$\varphi_2(\rho,c)-\frac{M}{6}\ge \varphi_2(\rho,1)-\frac{M}{6}=\frac{\rho^6+5\rho^5+3\rho^4+14\rho^3-3\rho^2+5\rho-1}{2(1-\rho)^4(1+\rho)^3}\ge0$$
for $\rho\ge 0.2002734$. The second part of the lemma follows.
\end{proof}

\begin{lemma} \label{r+h} For $r>3+2\sqrt2$ and  $\cos\theta\ge c$, we have 
$\left\vert h(\theta)\right\vert \le C_h\vert\theta\vert^3$ with
$$C_h=C_h(r,c):=\max\left(\frac{(1+\rho^2)(\rho^2+4\rho-1)}{6(1-\rho)^3(1+\rho)^2},\frac{(1+\rho^2)(1-2\rho(1+c)-\rho^2)}{6(1-\rho)((1+\rho^2)^2-4\rho^2c^2)}\right)\,.$$
Moreover, for $r\le 6.537$ and $c\ge1/2$, we have
$$C_h= \frac{(1+\rho^2)(\rho^2+4\rho-1)}{6(1-\rho)^3(1+\rho)^2}\,.$$
\end{lemma}

\begin{proof} 
We have
$$h(\theta)= \frac{f\left(\rho e^{i\theta}\right) -f\left(\rho e^{-i\theta}\right)}{2i}$$ 
so that $h(0)=0$ and
$$\aligned
h'(\theta)&=\frac{\rho e^{i\theta} f'\left(\rho e^{i\theta}\right) +\rho e^{-i\theta} f'\left(\rho e^{-i\theta}\right)}{2}=\Re\left( \rho e^{i\theta} f'\left(\rho e^{i\theta}\right)\right)\\
&=\Re\left(  \frac{r\rho e^{i\theta}}{1+\rho e^{i\theta}}-\frac{\rho e^{i\theta}}{1-\rho e^{i\theta}}-1\right)
= \frac{r\rho(\cos\theta+\rho)}{1+\rho^2+2\rho\cos\theta}-\frac{\rho(\cos\theta-\rho)}{1+\rho^2-2\rho\cos\theta}-1\,.
\endaligned$$
Since $h'(0)=0$, we find
$$\aligned
h'(\theta)&=\frac{r\rho(\cos\theta+\rho)}{1+\rho^2+2\rho\cos\theta}-\frac{r\rho}{1+\rho}
-\frac{\rho(\cos\theta-\rho)}{1+\rho^2-2\rho\cos\theta}+\frac{\rho}{1-\rho}\\
&= -\frac{r\rho(1-\rho)(1-\cos\theta)}{(1+\rho)(1+\rho^2+2\rho\cos\theta)}+\frac{\rho(1+\rho)(1-\cos\theta)}{(1-\rho)(1+\rho^2-2\rho\cos\theta)}\\
&=\left( \frac{\rho(1+\rho)}{(1-\rho)(1+\rho^2-2\rho\cos\theta)}-\frac{1}{1+\rho^2+2\rho\cos\theta}\right)\times(1-\cos\theta)\,.
\endaligned$$
Since
$$\displaylines{
\frac{\rho(1+\rho)}{(1-\rho)(1+\rho^2-2\rho\cos\theta)}-\frac{1}{1+\rho^2+2\rho\cos\theta}
\hfill\cr\hfill
\le \frac{\rho(1+\rho)}{(1-\rho)^3}-\frac{1}{(1+\rho)^2}=\frac{(1+\rho^2)(\rho^2+4\rho-1)}{(1-\rho)^3(1+\rho)^2}=: \varphi_3(\rho)\,,}$$
we deduce that $h'(\theta)\le \max(\varphi_3(\rho),0) \frac{\theta^2}{2}$.

Similarly we get $h'(\theta)\ge -\varphi_4(\rho,c)(1-\cos\theta)\ge-\max(\varphi_4(\rho,c),0)\frac{\theta^2}{2}$
with
$$\varphi_4(\rho,c):=\frac{1}{1+\rho^2+2\rho c}-\frac{\rho(1+\rho)}{(1-\rho)(1+\rho^2-2\rho c)}=\frac{(1+\rho^2)(1-2\rho(1+c)-\rho^2)}{(1-\rho)((1+\rho^2)^2-4\rho^2c^2)}\,.$$
We thus obtain
$$\left\vert h'(\theta)\right\vert \le\max\left( \varphi_3(\rho),0,\varphi_4(\rho,c)\right)\frac{\vert\theta\vert^2}{2}\,.$$
Since $-\varphi_4\le \varphi_3$, the first part of the lemma follows by integrating.

For $c\ge1/2$, note that
$$\varphi_3(\rho)-\varphi_4(\rho,c)\ge \varphi_3(\rho)-\varphi_4(\rho,1/2)
=\frac{(1+\rho^2)(2\rho^6+7\rho^5-3\rho^4-2\rho^3+3\rho^2+7\rho-2)}{(1-\rho)^3(1+\rho)^2(\rho^2+\rho+1)(\rho^2-\rho+1)}\ge0$$
for $\rho\ge 0.26101$. The second part of the lemma follows.
\end{proof}

\subsection{Proof of Theorem \ref{thas+}}

In view of Lemma \ref{key+}, we just need to estimate $K$, $C_g$, and $C_h$, when $\delta=\pi$.
Because of Lemma \ref{r+g}, we may choose $K=4/\pi^2$.
By Lemma \ref{r+g1}, and using the notation in its proof, we may choose  $C_g=\max\left(\frac{M}{24},\frac{\varphi_2(\rho,-1)}{4}\right)$. By \eqref{estM+} we already know $M/24<1/24$.
Since
$$\frac{\partial\varphi_2}{\partial\rho}(\rho,-1)=\frac{\rho^5+15\rho^4+10\rho^3+46\rho^2+17\rho+7}{(1+\rho)^2(1-\rho)^5}>0$$
and $\varphi_2(\sqrt2-1,-1)=(3+2\sqrt2)/2$, we find $C_g\le (3+2\sqrt2)/8$.

With the notation in the proof of Lemma \ref{r+h}, we may put  $C_h=\max(\varphi_3(\rho),\varphi_4(\rho,-1))/6$. Since
$$\varphi_3(\rho)-\varphi_4(\rho,-1)=\frac{2(1+\rho^2)(\rho^2+2\rho-1)}{(1-\rho)^3(1+\rho)^2}\le0$$
and
$$\frac{\partial\varphi_4}{\partial\rho}(\rho,-1)=\frac{1+5\rho+\rho^2+\rho^3}{6(1+\rho)^2(1-\rho)^3}\ge0\,,$$
we find $C_h\le \varphi_4(\sqrt2-1,-1)/6=(\sqrt2+1)/6$.

These estimates show the theorem.

\subsection{Proof of Theorem \ref{th0+}}

Let $\Phi_1(r,\lambda)$ denote the upper bound given in Theorem \ref{thas+}. It follows from \eqref{signM'} that $\Phi_1$ is a decreasing function of $r$.
The function $\Phi_1$ is also obviously a decreasing function of $\lambda$. We thus obtain
$$\Phi_1(r,\lambda)\le\Phi_1(5.941893,241)<0.9999978502\,,$$
for $r\ge5.941893$ and $\lambda\ge241$.
Theorem \ref{thas+} therefore implies that $I(r,\lambda)\neq0$ for $r\ge5.941893$ and $\lambda\ge241$.
Proposition \ref{firstvalues} ensures us that $I(r,\lambda)\neq0$ for any $r$ and $\lambda\le240$.
We thus get $I(r,\lambda)\neq0$ for any $r\ge5.941893$ and any positive integer $\lambda$.

We can now assume $r\le 5.941893$. Let us get a better version of Theorem \ref{thas+} in this case. By Lemma \ref{r+g}, we may choose $K=1$ when $\delta\le\frac{\pi}{3}$.
We also put $c=\cos\delta$ in the definitions of $C_g$ and $C_h$. We then get
$\left\vert  \frac{\sqrt{2\pi \lambda M}}{\exp(\lambda f(\rho))} I(\lambda)-1\right\vert\le \Phi_2(r,\lambda,\delta)$, where 
$$\Phi_2(r,\lambda,\delta):=\frac{3C_g(r,c)}{\lambda M^2}+\frac{15C_h(r,c)^2}{2\lambda M^3}
+\left( \frac{\sqrt2 }{\delta\sqrt{\pi\lambda M}} + \left(1-\frac{\delta}{\pi}\right)\sqrt{2\pi\lambda M} \right)\exp\left(-\lambda M\frac{\delta^2}{2}\right)\,.$$
Let us show that $\Phi_2$ is a decreasing function of $r$, to get results on an interval rather than at a point. Let us study each term defining $\Phi_2$.

We find
$$\frac{\partial C_g}{\partial\rho}(\rho,c)= \frac{P(\rho,c)+7(1-\rho)}{6(1+\rho)^2(1-\rho)^5(1+\rho^2+2\rho c)^2}\,,$$
where $P$ is a polynomial in $c$ and $\rho$ with nonnegative coefficients. We obtain $\frac{\partial C_g}{\partial\rho}(\rho,c)\ge0$, and the first term in $\Phi_2$ is therefore a decreasing function of $r$ by
\eqref{signrho'} and \eqref{signM'}.

We also find
$$\frac{\partial C_h}{\partial\rho}(\rho)=\frac{\rho^5+9\rho^4+8\rho^3+28\rho^2-\rho+3}{6(1-\rho)^4(1+\rho)^3}\ge0\,,$$
and the second term in $\Phi_2$ is also a decreasing function of $r$ by \eqref{signrho'} and \eqref{signM'}.

The third term is easily a a decreasing function of $r$ by  \eqref{signM'} when $\lambda M\delta^2\ge1$.
Therefore, for $\lambda M\delta^2\ge1$, the function $\Phi_2(r,\lambda,\delta)$ is a decreasing function of $r$ and $\lambda$, the monotonicity in $\lambda$ being easy under the condition  $\lambda M\delta^2\ge1$.
Computations then shows $\Phi_2(5.8478,241,0.75)<0.9936$ and $\Phi_2(5.8362,980,0.5)<0.9999$, while $\lambda M\delta^2>20$ in these two cases.
We thus proved $I(\lambda)\neq0$ for $r\ge 5.8362$, except when $241\le\lambda\le979$ and $5.8362\le r< 5.8478$. There are a finite number of possible exceptions, corresponding to the cases
$241\le\lambda_2\le979$ and $5.8362\lambda_2\le \lambda_1< 5.8478\lambda_2$. We checked these cases with {\tt Maple} in 5470 seconds.
Here the limitation comes from the size of $\lambda_2$ that should be handled by {\tt Maple} in the computations. For our program, the limitation is  
$\lambda_2\le998$.

\section{The case $r< 3+2\sqrt2$}

\subsection{General properties}

In this case, the derivative $f'$ has exactly two zeroes. These zeroes are the conjugate complex numbers
$$z_r=\frac{r-1+ i\sqrt{-r^2+6r-1}}{2r}=\rho e^{i\alpha} \quad\text{and}\quad
\bar z_r=\frac{r-1- i\sqrt{-r^2+6r-1}}{2r} =\rho e^{-i\alpha}\,,$$
 with $\rho=\frac{1}{\sqrt r}$ and $\alpha\in(0,\pi/2)$.
Note that 
\begin{equation}\label{alpha}
\cos\alpha =\frac{r-1}{2\sqrt r} \quad\text{and}\quad  \sin\alpha=\frac{\sqrt{-r^2+6r-1}}{2\sqrt r}\,.
\end{equation}

Many properties from the previous section can be rewritten. We still have
$$r=\frac{1+z_r}{z_r(1-z_r)}\quad\text{and}\quad
f''(z_r)=\frac{1-2z_r-z_r^2}{(1+z_r)z_r^2(1-z_r)^2}\,,$$
from which we find
\begin{equation}\label{exprf''}
f''(z_r)z_r^2=\frac{\sqrt{-r^2+6r-1}}{2}\left(\frac{(r+1)\sqrt{-r^2+6r-1}-i(r-1)^2}{4r}\right)=\frac{\sqrt{-r^2+6r-1}}{2} e^{-i\beta}
\end{equation}
with $\beta\in(0,\pi/2)$. 
Note that 
\begin{equation}\label{beta}
\cos\beta=\frac{(r+1)\sqrt{-r^2+6r-1}}{4r} \quad\text{and}\quad \sin\beta=\frac{(r-1)^2}{4r}\,.
\end{equation}

We also use \eqref{deff'} to compute
$$\frac{\partial}{\partial\theta}\left( f\left( \rho e^{i\theta}\right)\right) =i\frac{-r\rho^2e^{2i\theta}+(r-1)\rho e^{i\theta}-1}{1-\rho^2 e^{2i\theta}}
=ir\frac{-e^{2i\theta}+2e^{i\theta}\cos\alpha -1}{r-e^{2i\theta}}\,.$$
We then obtain the useful expressions
\begin{equation}\label{fg}
g'(\theta)+ih'(\theta)=ir\frac{(e^{i\theta}-e^{i\alpha})(e^{-i\alpha}-e^{i\theta})}{r-e^{2i\theta}}=2ir\frac{\cos\alpha-\cos\theta}{(r-1)\cos\theta-i(r+1)\sin\theta}\,.
\end{equation}

As in the previous section, we now state our key lemma.

\begin{lemma}\label{key-} Let $\delta\in\lbrack0,\alpha\rbrack$. Assume that
\begin{enumerate}
\item $\left\vert   f\left( \rho e^{i\theta}\right)- f\left( \rho e^{i\alpha}\right) +\frac{\sqrt{-r^2+6r-1}}{4} e^{-i\beta}(\theta-\alpha)^2 \right\vert \le C_f\vert \theta-\alpha\vert^3$, for some positive constant $C_f$,
\item $g(\theta)-g(\alpha)\le -\frac{(r+1)(-r^2+6r-1)}{16r}(\theta-\alpha)^2+C_g\vert\theta-\alpha\vert^3$, for some constant $C_g>0$,
\end{enumerate}
for $\alpha-\delta\le\theta\le\alpha+\delta$. 

We then have
$$ \displaylines{\left\vert \frac{(-r^2+6r-1)^{1/4}\sqrt{\pi\lambda}}{2^{1+\frac{r+1}{2}\lambda}} I(\lambda)-\cos\left( (r\gamma_1+\gamma_2)\lambda+\frac{\beta}{2}\right) \right\vert
\le\frac{128r^2C_f e^{C_g\lambda\delta^3}}{\sqrt{\pi\lambda}(r+1)^2(-r^2+6r-1)^{7/4}}
\hfill\cr\hfill
+\frac{8re^{ -\frac{(r+1)(-r^2+6r-1)}{16r} \lambda\delta^2}}{\delta\sqrt{\pi\lambda}(r+1)(-r^2+6r-1)^{3/4} }
+ \frac{(\pi-2\delta)(-r^2+6r-1)^{1/4}\sqrt{\lambda}}{2\sqrt\pi} e^{-\frac{(r+1)(-r^2+6r-1)}{16r}\lambda\delta^2+C_g\lambda\delta^3}
\,,}$$
with $\gamma_1=\arccos\frac{3r-1}{2\sqrt2r}$, and $\gamma_2=-\arccos\frac{r-3}{2\sqrt2}$.
\end{lemma}

\begin{proof} Instead of \eqref{expcos}, we use the bound 
\begin{equation}\label{exp}
\left\vert e^z-1\right\vert\le  \vert z\vert e^{\max(\Re z,0)}\,,
\end{equation}
with $z= f\left( \rho e^{i\theta}\right)- f\left( \rho e^{i\alpha}\right) +\frac{\sqrt{-r^2+6r-1}}{4} e^{-i\beta}(\theta-\alpha)^2$. 
Conditions 1 and 2 then imply $\vert z\vert\le C_f\lambda\vert \theta-\alpha\vert^3$ and $\Re z\le C_g\lambda\vert\theta-\alpha\vert^3$, using  \eqref{beta}.
We thus find
$$\displaylines{
\left\vert \int_{\alpha-\delta}^{\alpha+\delta} e^{\lambda\left(f\left(\rho e^{i\theta}\right) - f\left(\rho e^{i\alpha}\right) \right)}\frac{d\theta}{2\pi}
- \int_{\alpha-\delta}^{\alpha+\delta} e^{ -\frac{\sqrt{-r^2+6r-1}}{4} e^{-i\beta}\lambda(\theta-\alpha)^2  }\frac{d\theta}{2\pi} \right\vert
\hfill\cr\hfill
\begin{aligned}
&=\left\vert \int_{\alpha-\delta}^{\alpha+\delta} (e^z-1)e^{ -\frac{\sqrt{-r^2+6r-1}}{4} e^{-i\beta}\lambda(\theta-\alpha)^2  }\frac{d\theta}{2\pi} \right\vert\\
&\le \int_{\alpha-\delta}^{\alpha+\delta} C_f\lambda\vert \theta-\alpha\vert^3 \exp\left(\lambda\max\left(g(\theta)-g(\alpha), -\frac{\sqrt{-r^2+6r-1}}{4} \cos\beta (\theta-\alpha)^2 \right)\right) \frac{d\theta}{2\pi}\\
&\le C_f\lambda \int_{-\delta}^{\delta} \vert u\vert^3 \exp\left(-\frac{(r+1)(-r^2+6r-1)}{16r} \lambda u^2+C_g\lambda \vert u\vert^3\right) \frac{du}{2\pi}\\
&\le C_f\lambda e^{C_g\lambda\delta^3} \int_0^\pi u^3 \exp\left(-\frac{(r+1)(-r^2+6r-1)}{16r} \lambda u^2\right) \frac{du}{\pi}\\
&= C_f e^{C_g\lambda\delta^3} \frac{128r^2}{\pi(r+1)^2(-r^2+6r-1)^2\lambda}
\,.\end{aligned}}$$
Since
$$\begin{aligned}
\left\vert \ \int_{\vert \theta-\alpha\vert>\delta} e^{ -\frac{\sqrt{-r^2+6r-1}}{4} e^{-i\beta}\lambda(\theta-\alpha)^2  }\frac{d\theta}{2\pi} \right\vert
&\le \int_{\delta}^{+\infty} \frac{u}{\delta} e^{ -\frac{\sqrt{-r^2+6r-1}}{4}\cos\beta  \lambda u^2  }\frac{du}{\pi}\\
&=\frac{8re^{ -\frac{(r+1)(-r^2+6r-1)}{16r} \lambda \delta^2}}{\pi\delta\lambda(r+1)(-r^2+6r-1) }\,.
\end{aligned}$$
and 
$$\int_{-\infty}^{+\infty} e^{ -\frac{\sqrt{-r^2+6r-1}}{4} e^{-i\beta}\lambda(\theta-\alpha)^2  }\frac{d\theta}{2\pi} 
=\frac{e^{i\beta/2}}{(-r^2+6r-1)^{1/4} \sqrt{\pi\lambda}}$$
we get
$$\displaylines{
\left\vert\frac{(-r^2+6r-1)^{1/4} \sqrt{\pi\lambda}}{e^{\lambda g(\alpha)}} \int_{\alpha-\delta}^{\alpha+\delta} e^{\lambda f\left(\rho e^{i\theta}\right)}\frac{d\theta}{2\pi}- e^{i\lambda h(\alpha)+i\beta/2}\right\vert
\hfill\cr\hfill
\le  \frac{128r^2C_f e^{C_g\lambda\delta^3}}{\sqrt{\pi\lambda}(r+1)^2(-r^2+6r-1)^{7/4}}+\frac{8re^{ -\frac{(r+1)(-r^2+6r-1)}{16r} \lambda \delta^2}}{\delta\sqrt{\pi\lambda}(r+1)(-r^2+6r-1)^{3/4} }
\,.}$$
By \eqref{fg}  we have
$$g'(\theta)=-\frac{2r(r+1)\sin\theta}{\vert (r-1)\cos\theta+i(r+1)\sin\theta\vert^2}(\cos\alpha-\cos\theta)\,,$$
which shows that $g$ is increasing on $\lbrack0,\alpha\rbrack$ and decreasing on $\lbrack\alpha,\pi\rbrack$. We deduce from Condition 2
$$\begin{aligned}
\left\vert \int_0^{\alpha-\delta} e^{\lambda\left(f\left(\rho e^{i\theta}\right) - f\left(\rho e^{i\alpha}\right)\right)}\frac{d\theta}{2\pi}\right\vert
&\le  \int_0^{\alpha-\delta} e^{\lambda(g(\theta)-g(\alpha))} \frac{d\theta}{2\pi} 
\le \frac{\alpha-\delta}{2\pi} e^{\lambda(g(\alpha-\delta)-g(\alpha))}\\
&\le \frac{\alpha-\delta}{2\pi} e^{  -\frac{(r+1)(-r^2+6r-1)}{16r}\lambda\delta^2+C_g\lambda\delta^3}
\end{aligned}$$
and
$$\left\vert \int_{\alpha+\delta}^{\pi} e^{\lambda\left(f\left(\rho e^{i\theta}\right) - f\left(\rho e^{i\alpha}\right)\right)}\frac{d\theta}{2\pi}\right\vert
\le \frac{\pi-\alpha-\delta}{2\pi} e^{ -\frac{(r+1)(-r^2+6r-1)}{16r}\lambda\delta^2+C_g\lambda\delta^3}\,.$$
We deduce
$$\displaylines{
\left\vert\frac{(-r^2+6r-1)^{1/4} \sqrt{\pi\lambda}}{e^{\lambda g(\alpha)}} \int_0^{\pi} e^{\lambda f\left(\rho e^{i\theta}\right)}\frac{d\theta}{2\pi}- e^{i\lambda h(\alpha)+i\beta/2}\right\vert
\le  \frac{128r^2C_f e^{C_g\lambda\delta^3}}{\sqrt{\pi\lambda}(r+1)^2(-r^2+6r-1)^{7/4}}
\hfill\cr\hfill
+\frac{8re^{ -\frac{(r+1)(-r^2+6r-1)}{16r} \lambda\delta^2}}{\delta\sqrt{\pi\lambda}(r+1)(-r^2+6r-1)^{3/4} }
+ \frac{(\pi-2\delta)(-r^2+6r-1)^{1/4}\sqrt{\lambda}}{2\sqrt\pi} e^{-\frac{(r+1)(-r^2+6r-1)}{16r}\lambda\delta^2+C_g\lambda\delta^3}
\,.}$$

By \eqref{equa1}  we have
$ I(\lambda)=2\Re\left(  \int_0^{\pi} e^{\lambda f\left(\rho e^{i\theta}\right)}\frac{d\theta}{2\pi} \right)$.
and the lemma then follows from \eqref{beta} and
$$\begin{aligned}
e^{i\gamma_1}&=\frac{1+z_r}{\vert 1+z_r\vert}=\frac{3r-1+i\sqrt{-r^2+6r-1}}{2\sqrt2r}\,,\\
e^{i\gamma_2}&=\frac{z_r^{-1}-1}{\vert z_r^{-1}-1\vert}=\frac{r-3-i\sqrt{-r^2+6r-1}}{2\sqrt2}\,,\\
e^{\lambda g(\alpha)}&=\vert 1+z_r\vert^{r\lambda}\vert z_r^{-1}-1\vert^{\lambda}=\sqrt2^{(r+1)\lambda}\,.
\end{aligned}$$

\end{proof}

\subsection{Estimates for $f$ and $g$}

\begin{lemma}\label{r-f}
For $\alpha/2\le\theta\le \pi-\alpha/2$, we have
$$\left\vert  f\left( \rho e^{i\theta}\right)- f\left( \rho e^{i\alpha}\right) +\frac{\sqrt{-r^2+6r-1}}{4} e^{-i\beta}(\theta-\alpha)^2\right\vert\le 
0.33846\frac{(r+1)^2}{r^2}\left\vert \theta-\alpha\right\vert^3\,.$$
\end{lemma}

\begin{proof}
From \eqref{fg} we know that
$$g'(\theta)+ih'(\theta)=ir\frac{(e^{i\theta}-e^{i\alpha})(e^{-i\alpha}-e^{i\theta})}{r-e^{2i\theta}}\,.$$
We also have
$$\frac{e^{-i\alpha}-e^{i\theta}}{r-e^{2i\theta}}=\frac{e^{-i\alpha}-e^{i\alpha}}{r-e^{2i\alpha}}
+\frac{(e^{i\theta}-e^{i\alpha})(e^{i(\theta-\alpha)}+1-r-e^{i(\alpha+\theta)})}{(r-e^{2i\theta})(r-e^{2i\alpha})}\,,$$
which gives
$$\left\vert g'(\theta)+ih'(\theta) -ir\frac{e^{-i\alpha}-e^{i\alpha}}{r-e^{2i\alpha}}(e^{i\theta}-e^{i\alpha})\right\vert
\le \frac{\sqrt{-r^2+6r-1}+\sqrt r (r-1)}{2\sqrt{r^2+1+\sqrt r (r-1)}}(\theta-\alpha)^2$$
since
$$\begin{aligned}
\vert e^{i(\theta-\alpha)}+1-r-e^{i(\alpha+\theta)}\vert&\le r-1+2\sin\alpha= r-1+\frac{\sqrt{-r^2+6r-1}}{\sqrt r}\,,\\
\vert r-e^{2i\alpha}\vert^2=r^2+1-2r\cos(2\alpha)&=r^2+1-2r\left(2\frac{(r-1)^2}{4r}-1\right)=4r\,,\\
\vert r-e^{2i\theta}\vert^2=r^2+1-2r\cos(2\theta)&\ge r^2+1-2r\cos(\alpha)= r^2+1-\sqrt r (r-1)\,.
\end{aligned}$$
Moreover we find
$$\left\vert \frac{e^{-i\alpha}-e^{i\alpha}}{r-e^{2i\alpha}}(e^{i\theta}-e^{i\alpha}) - \frac{i(1-e^{2i\alpha})}{r-e^{2i\alpha}}(\theta-\alpha)\right\vert
\le  \frac{\sin\alpha}{\vert r-e^{2i\alpha}\vert}(\theta-\alpha)^2=\frac{\sqrt{-r^2+6r-1}}{4r}(\theta-\alpha)^2$$
and therefore
$$\displaylines{\left\vert g'(\theta)+ih'(\theta) +r\frac{1-e^{2i\alpha}}{r-e^{2i\alpha}}(\theta-\alpha)\right\vert
\hfill\cr\hfill
\le \left( \frac{\sqrt{-r^2+6r-1}+\sqrt r (r-1)}{2\sqrt{r^2+1+\sqrt r (r-1)}}+\frac{\sqrt{-r^2+6r-1}}{4} \right)(\theta-\alpha)^2\,.}$$
We deduce the lemma by using  \eqref{alpha}, \eqref{exprf''}, \eqref{beta}, and checking
$$\frac{\sqrt{-r^2+6r-1}+\sqrt r (r-1)}{2\sqrt{r^2+1+\sqrt r (r-1)}}+\frac{\sqrt{-r^2+6r-1}}{4}\le 1.01537\frac{(r+1)^2}{r^2}\,.$$
\end{proof}

\begin{lemma}\label{r-g} For any $\alpha/2\le\theta\le \pi-\alpha/2$, we have
$$g(\theta)-g(\alpha) \le -\frac{(r+1)(-r^2+6r-1)}{16r}(\theta-\alpha)^2+\frac{r+1}{4}\vert\theta-\alpha\vert^3\,.$$
\end{lemma}

\begin{proof} We find
$$\begin{aligned}
g(\theta)-g(\alpha) &= \frac{r}{2}\log\left( \frac{1+\rho^2+2\rho\cos\theta}{1+\rho^2+2\rho\cos\alpha} \right) 
+\frac{1}{2} \log\left( \frac{1+\rho^2-2\rho\cos\theta}{1+\rho^2-2\rho\cos\alpha} \right)\\
&=\frac{r}{2}\log\left( 1+\frac{2\rho(\cos\theta-\cos\alpha)}{1+\rho^2+2\rho\cos\alpha} \right)
+\frac{1}{2} \log\left(1-\frac{2\rho(\cos\theta-\cos\alpha)}{1+\rho^2-2\rho\cos\alpha} \right)\\
&=\frac{r}{2}\log\left( 1+ \rho(\cos\theta-\cos\alpha)\right) +\frac{1}{2} \log\left(1-\rho^{-1}(\cos\theta-\cos\alpha) \right)
\end{aligned}$$
since $2\rho\cos\alpha=1-\rho^2$. We deduce the upper bound
$$\begin{aligned}
g(\theta)-g(\alpha)&\le -\frac{r+1}{4}(\cos\theta-\cos\alpha)^2+\frac{1-r^2}{6\sqrt r}(\cos\theta-\cos\alpha)^3\\
&=\frac{r+1}{4}\left( -(\cos\theta-\cos\alpha)^2-\frac{4\cos\alpha}{3}(\cos\theta-\cos\alpha)^3\right)\,.
\end{aligned}$$
Let us define $\varphi_5(\theta)=\sin^2\alpha (\theta-\alpha)^2-(\cos\theta-\cos\alpha)^2-\frac{4\cos\alpha}{3}(\cos\theta-\cos\alpha)^3$, so that
$$\begin{aligned}
\varphi_5''(\theta)&=2(\cos\theta-\cos\alpha)\left( 6\cos\alpha \cos^2\theta  +2(1-\cos^2\alpha) \cos\theta-3\cos\alpha \right)\\
&=: 2(\cos\theta-\cos\alpha)p(\cos\theta,\cos\alpha)\,.\end{aligned}$$
We check that
$$\begin{aligned}
p(\cos\theta,\cos\alpha)&\le p(\cos(\alpha/2),\cos\alpha)=3+(1-\cos^2\alpha)\left(\sqrt{2(1+\cos\alpha)}-3\right)\le 3\,,\\
p(\cos\theta,\cos\alpha)&\ge \min_{-1\le x\le 1\atop 0\le y\le 1}p(x,y)=p(0,1)= -3\,.
\end{aligned}$$
This gives $\vert \varphi_5''(\theta)\vert\le 6\vert \theta-\alpha\vert$ and
$ \varphi_5(\theta)\le \vert \theta-\alpha\vert^3$, and the lemma follows using \eqref{alpha}.
\end{proof}

\subsection{Proof of Theorem \ref{thas-}}

We use Lemmas \ref{key-}, \ref{r-f}, \ref{r-g} to get
$$ \displaylines{\left\vert \frac{(-r^2+6r-1)^{1/4}\sqrt{\pi\lambda}}{2^{1+\frac{r+1}{2}\lambda}} I(\lambda)-\cos\left( (r\gamma_1+\gamma_2)\lambda_2+\frac{\beta}{2}\right) \right\vert
\le\frac{24.45 e^{(r+1)\lambda\frac{\delta^3}{4}}}{\sqrt{\lambda}(-r^2+6r-1)^{7/4}}
\hfill\cr\hfill
+\frac{8re^{ -\frac{(r+1)(-r^2+6r-1)}{16r} \lambda\delta^2}}{\delta\sqrt{\pi\lambda}(r+1)(-r^2+6r-1)^{3/4} }
+ \frac{(\pi-2\delta)(-r^2+6r-1)^{1/4}\sqrt{\lambda}}{2\sqrt\pi} e^{-\frac{(r+1)(-r^2+6r-1)}{16r}\lambda\delta^2+(r+1)\lambda\frac{\delta^3}{4}}
\,,}$$
for $\delta\le\alpha/2$. We apply the inequalities $xe^{-x}\le1$ and $x^3e^{-x}\le27e^{-3}$ to the second and third term of the right hand side respectively to get
$$ \displaylines{\left\vert \frac{(-r^2+6r-1)^{1/4}\sqrt{\pi\lambda}}{2^{1+\frac{r+1}{2}\lambda}} I(\lambda)-\cos\left( (r\gamma_1+\gamma_2)\lambda_2+\frac{\beta}{2}\right) \right\vert
\le\frac{24.45 e^{(r+1)\lambda\frac{\delta^3}{4}}}{\sqrt{\lambda}(-r^2+6r-1)^{7/4}}
\hfill\cr\hfill
+\frac{128r^2}{\delta^3\lambda^{3/2}\sqrt{\pi}(r+1)^2(-r^2+6r-1)^{7/4} }
+ \frac{ 55296\sqrt\pi r^3 e^{(r+1)\lambda\frac{\delta^3}{4}}}{ e^3\delta^6\lambda^{5/2} (r+1)^3(-r^2+6r-1)^{11/4} } 
\,.}$$
We now choose $\delta=\left(\frac{(r+1)\lambda}{8}\right)^{-1/3}$. By hypothesis, we have $\delta\le \frac{\sin\alpha}{2}\le \frac{\alpha}{2}$. We obtain
$$ \displaylines{\left\vert \frac{(-r^2+6r-1)^{1/4}\sqrt{\pi\lambda}}{2^{1+\frac{r+1}{2}\lambda}} I(\lambda)-\cos\left( (r\gamma_1+\gamma_2)\lambda_2+\frac{\beta}{2}\right) \right\vert
\hfill\cr\hfill
\begin{aligned}
&\le \frac{1}{\sqrt\lambda (-r^2+6r-1)^{11/4}}\left( 24.45e^2(-r^2+6r-1)+\frac{16r^2(-r^2+6r-1)}{\sqrt\pi(r+1)}+\frac{864\sqrt\pi r^3}{e(r+1)}\right)\\
&\le \frac{16336}{\sqrt\lambda (-r^2+6r-1)^{11/4}}\,,
\end{aligned}}$$
which proves the theorem.

\subsection{Proof of Theorem \ref{thexceptions}}

For $r>3+2\sqrt2$, Theorem \ref{thas+} shows that $I(\lambda)\neq0$ for $\lambda$ large enough, the implied bound only depending on $M(r)$. The first part of the theorem follows.

For $1<r<3+2\sqrt2$, define $\mathcal S_r$ as the set of non negative integers $\lambda$ such that $I(r,\lambda)=0$. As usual, let $\Vert x\Vert$ denote the distance of $x$ to the nearest integer.
For $\lambda\in\mathcal S_r$, we have
$$\left\Vert (r\gamma_1+\gamma_2)\lambda+\frac{\pi+\beta}{2}\right\Vert\le \frac{\pi}{2}\left\vert \sin\left(  (r\gamma_1+\gamma_2)\lambda+\frac{\pi+\beta}{2}\right)\right\vert\le \frac{25661}{\sqrt\lambda (-r^2+6r-1)^{11/4}}\,,$$
by Theorem \ref{thas-}.  For $\lambda,\lambda'\in\mathcal S_r$, $\lambda<\lambda'<\lambda+\frac{(-r^2+6r-1)^{11/4} \sqrt\lambda }{102644}$, we find
$$\left\Vert (r\gamma_1+\gamma_2)(\lambda'-\lambda)\right\Vert\le \frac{51322}{\sqrt\lambda (-r^2+6r-1)^{11/4}}\le\frac{1}{2(\lambda'-\lambda)}\,.$$
By Legendre's theorem \cite{L}, this implies that $\lambda'-\lambda$ is a denominator $q_n$ in the continued fraction expansion of $r\gamma_1+\gamma_2$.
Since $q_n\ge F_{n+1}=\frac{1}{\sqrt5}\left( \left(\frac{1+\sqrt5}{2}\right)^{n+1}-\left(\frac{1-\sqrt5}{2}\right)^{n+1}\right)\ge \frac{1}{\sqrt5} \left(\frac{1+\sqrt5}{2}\right)^n$,
the number of such denominators less than $q$ is upper bounded by $\frac{\log (q\sqrt5)}{\log\left(\frac{1+\sqrt5}{2}\right)}$, and we get for $x\ge1$:
$$\displaylines{
\#\mathcal S_r\cap\left\lbrack x,x+\frac{(-r^2+6r-1)^{11/4}\sqrt x}{102644}\right\rbrack
\hfill\cr\hfill
\le \frac{\log\frac{ (-r^2+6r-1)^{11/4}\sqrt{5x}}{102644} }{\log\left(\frac{1+\sqrt5}{2}\right)}=\varphi_6\left(x+\frac{(-r^2+6r-1)^{11/4}\sqrt x}{102644} 
\right)-\varphi_6(x)\,,}$$
with
$$\varphi_6(x)=\frac{102644}{ (-r^2+6r-1)^{11/4}\log\left(\frac{1+\sqrt5}{2}\right)}\sqrt x\log x+O_r(\sqrt x)\,.$$
The second part of the theorem follows.

\section{The case $r$ close to $1$}

When $r$ goes to $1$, the angles $\alpha$ and $\beta$ go to $\pi/2$ and $0$ respectively. So we shall prove specific estimates in this case.
Before that, we establish Proposition \ref{propr1}.

\subsection{Small values of $\lambda_1-\lambda_2$}

We have
$$\begin{aligned}
C_{\lambda_1,\lambda_2}(-1,1)&=\oint (z+1)^{\lambda_1}(1-z)^{\lambda_2}\frac{dz}{2i\pi z^{\lambda_1+1}}
=\oint (1-z^2)^{\lambda_2}(z+1)^{\lambda_1-\lambda_2} \frac{dz}{2i\pi z^{\lambda_1+1}}\\
&=\sum_{\lambda_2\le 2j\le \lambda_1} (-1)^j{\lambda_2\choose j}{\lambda_1-\lambda_2\choose \lambda_1-2j}\\
&=  \frac{\lambda_2!}{  \left\lfloor\frac{\lambda_1}{2}\right\rfloor! \left\lfloor\frac{\lambda_2}{2}\right\rfloor!} 
\sum_{\left\lceil\frac{\lambda_2}{2}\right\rceil \le j\le \left\lfloor\frac{\lambda_1}{2}\right\rfloor} (-1)^j {\lambda_1-\lambda_2\choose \lambda_1-2j}
\frac{\left\lfloor\frac{\lambda_1}{2}\right\rfloor!}{j!}\frac{\left\lfloor\frac{\lambda_2}{2}\right\rfloor!}{(\lambda_2-j)!} \\
&=  \frac{\lambda_2! (-1)^{\left\lceil\frac{\lambda_2}{2}\right\rceil}}{  \left\lfloor\frac{\lambda_1}{2}\right\rfloor! \left\lfloor\frac{\lambda_2}{2}\right\rfloor!} 
\sum_{0 \le j\le \left\lfloor\frac{\lambda_1}{2}\right\rfloor-\left\lceil\frac{\lambda_2}{2}\right\rceil}
(-1)^j {\lambda_1-\lambda_2\choose 2j+\left\lceil\frac{\lambda_2}{2}\right\rceil-\left\lfloor\frac{\lambda_2}{2}\right\rfloor}
\frac{\left\lfloor\frac{\lambda_1}{2}\right\rfloor!}{(\left\lceil\frac{\lambda_2}{2}\right\rceil+j)!}
\frac{\left\lfloor\frac{\lambda_2}{2}\right\rfloor!}{(\left\lfloor\frac{\lambda_2}{2}\right\rfloor-j)!} 
\,.\end{aligned}$$
When $\lambda_1-\lambda_2$ is fixed, the last sum is  a polynomial $\widetilde C$ in $\left\lfloor\frac{\lambda_2}{2}\right\rfloor$ of degree at most
$\left\lfloor\frac{\lambda_1}{2}\right\rfloor-\left\lceil\frac{\lambda_2}{2}\right\rceil$, also depending on the parity of $\lambda_1$ and $\lambda_2$. More precisely we consider the four families of polynomials
$$\begin{aligned}
\widetilde C_{l,0,0}(k)&=\sum_{0\le j\le l} (-1)^j{2l\choose 2j}\frac{(k+l)!}{(k+j)!}\frac{k!}{(k-j)!}                 &\text{for $(\lambda_2,\lambda_1)=(2k,2k+2l)$,}\\
\widetilde C_{l,0,1}(k)&=\sum_{0\le j\le l} (-1)^j{2l+1\choose 2j}\frac{(k+l)!}{(k+j)!}\frac{k!}{(k-j)!}             &\text{for $(\lambda_2,\lambda_1)=(2k,2k+2l+1)$,}\\
\widetilde C_{l,1,0}(k)&=\sum_{0\le j\le l-1} (-1)^j{2l\choose 2j+1}\frac{(k+l)!}{(k+1+j)!}\frac{k!}{(k-j)!}      &\text{for $(\lambda_2,\lambda_1)=(2k+1,2k+2l+1)$,}\\
\widetilde C_{l,1,1}(k)&=\sum_{0\le j\le l} (-1)^j{2l+1\choose 2j+1}\frac{(k+l+1)!}{(k+1+j)!}\frac{k!}{(k-j)!} &\text{for $(\lambda_2,\lambda_1)=(2k+1,2k+2l+2)$.}
\end{aligned}$$
The first values are given by $\widetilde C_{0,0,0}(k)=\widetilde C_{1,0,0}(k)=\widetilde C_{0,0,1}(k)=\widetilde C_{0,1,1}(k)=1$, $\widetilde C_{0,1,0}(k)=0$, $\widetilde C_{1,1,0}(k)=2$, $\widetilde C_{2,1,0}(k)=8$, 
$\widetilde C_{1,0,1}(k)=-2k+1$, $\widetilde C_{1,1,1}(k)=2k+6$.
We can compute the leading terms of each polynomial to show that the degree of $\widetilde C_{l,0,0}(k)$, $\widetilde C_{l,0,1}(k)$, $\widetilde C_{l,1,1}(k)$ is at least $2$ for $l\ge2$, and that the degree of $\widetilde C_{l,1,0}(k)$ is at least $2$ for $l\ge3$:
$$\begin{aligned}
\widetilde C_{l,0,0}(k)&=\sum_{0\le j\le l} (-1)^i{2l\choose 2j} \left( k^l+\left({l+1\choose2}-j^2\right)k^{l-1}+\cdots\right)\\
&=\Re\left(\sum_{0\le m\le 2l} i^m {2l\choose m}  \left( k^l+\left({l+1\choose2}-\frac{m^2}{4}\right)k^{l-1}+\cdots\right)\right)\\
&=\Re\left( (1+i)^{2l} k^l+\left( {l+1\choose2} -\frac{l(1+2li)}{4} \right)(1+i)^{2l} k^{l-1}+\cdots\right)\\
\widetilde C_{l,0,1}(k)&=\sum_{0\le j\le l} (-1)^j{2l+1\choose 2j} \left( k^l+\cdots \right)=\Re\left(\sum_{0\le m\le 2l+1} i^m {2l+1\choose m}  \left( k^l+\cdots\right)\right)\\
&=\Re\left( (1+i)^{2l+1} k^l+\cdots\right)=\pm 2^lk^l+\cdots\\
\widetilde C_{l,1,0}(k)&=\sum_{0\le j\le l-1} (-1)^j{2l\choose 2j+1} \left(  k^{l-1}+\left({l+1\choose2}-j^2-j-1\right)k^{l-2}+\cdots \right)\\
&=\Im\left(\sum_{0\le m\le 2l} i^m {2l\choose m}  \left( k^{l-1}+\left({l+1\choose2}-\frac{m^2+3}{4}\right)k^{l-2}+\cdots\right)\right)\\
&=\Im\left( (1+i)^{2l} k^{l-1}+\left( {l+1\choose2} -\frac{3+l(1+2li)}{4} \right)(1+i)^{2l} k^{l-2}+\cdots\right)\\
\widetilde C_{l,1,1}(k)&=\sum_{0\le j\le l} (-1)^j{2l+1\choose 2j+1} \left(k^l+\cdots \right)=\Im\left(\sum_{0\le m\le 2l+1} i^m {2l+1\choose m}  \left( k^l+\cdots\right)\right)\\
&=\Im\left( (1+i)^{2l+1} k^l+\cdots\right)=\pm 2^lk^l+\cdots
\,.\end{aligned}$$

For each of the four cases, we checked with {\tt Maple} that $\widetilde C_l$ is an irreducible polynomial for $2\le l\le350$, which proves Proposition \ref{propr1}
since $\widetilde C_l$ then has no integer zero. Each case required between $95000$ and $98000$ seconds.

\subsection{The approach}

We give a more specific version of Lemma \ref{key-}.

\begin{lemma}\label{key1} Let $\delta,\delta_1$ be two positive real numbers with  $\delta<\min\left(2/3,3(3-r)/4\right)$ and $\delta\le \delta_1\le\alpha$. Assume that
\begin{enumerate}
\item $\left\vert  f\left( \rho e^{i\theta}\right)- f\left( \rho e^{i\alpha}\right) +\frac{(\theta-\alpha)^2}{2} \right\vert
\le \frac{\vert\theta-\alpha\vert^3}{3}+\frac{r-1}{4} (\theta-\alpha)^2$, for $\alpha-\delta\le\theta\le\alpha+\delta$,
\item $g(\theta)-g(\alpha)\le -\frac{(\theta-\alpha)^2}{2}+\frac{\vert\theta-\alpha\vert^3}{2} $, for $\alpha-\delta_1\le\theta\le\alpha+\delta_1$.
\end{enumerate}
Define $\gamma_1=\arccos\frac{3r-1}{2\sqrt2r}$ and $\gamma_2=-\arccos\frac{r-3}{2\sqrt2}$.
We then have
$$ \displaylines{\left\vert \frac{\sqrt{\pi\lambda}}{2^{\frac{1+(r+1)\lambda}{2}}} I(\lambda)-\cos\left( (r\gamma_1+\gamma_2)\lambda\right) \right\vert
\le \frac{1}{3\sqrt{2\pi\lambda}\left(\frac{3-r}{4}-\frac{\delta}{3}\right)^2}+\frac{r-1}{2^{7/2}\left(\frac{3-r}{4}-\frac{\delta}{3}\right)^{3/2}}
+\frac{\sqrt2e^{ -\lambda\frac{\delta^2}{2}  }}{\sqrt{\pi\lambda}\delta}
\hfill\cr\hfill
+\left(\frac{\pi-2\delta_1}{\sqrt{2\pi}} + \frac{2^{3/2}}{\sqrt\pi\delta(2-3\delta)}  \right)\sqrt\lambda e^{-\lambda\frac{\delta_1^2}{2}+\lambda\frac{\delta_1^3}{2}}\,.}$$
\end{lemma}

\begin{proof} We follow the proof of Lemma \ref{key-}, and use Conditions 1 and 2:
$$\displaylines{
\left\vert \int_{\alpha-\delta}^{\alpha+\delta}e^{\lambda\left(  f\left(\rho e^{i\theta}\right) - f\left(\rho e^{i\alpha}\right) \right)} \frac{d\theta}{2\pi}
- \int_{\alpha-\delta}^{\alpha+\delta} e^{ -\lambda\frac{(\theta-\alpha)^2}{2}} \frac{d\theta}{2\pi} \right\vert
\hfill\cr\hfill
\begin{aligned}
&\le  \int_{\alpha-\delta}^{\alpha+\delta} \lambda\left(\frac{\vert\theta-\alpha\vert^3}{3}+\frac{r-1}{4} (\theta-\alpha)^2\right)
e^{-\frac{3-r}{4}\lambda(\theta-\alpha)^2+\lambda\frac{\vert\theta-\alpha\vert^3}{3}}  \frac{d\theta}{2\pi} \\
&\le \lambda \int_0^{\delta} \left( \frac{u^3}{3}+\frac{r-1}{4} u^2\right)e^{-\lambda u^2\left(\frac{3-r}{4}-\frac{\delta}{3}\right)} \frac{du}{\pi} 
<\lambda \int_0^{\infty} \left( \frac{u^3}{3}+\frac{r-1}{4} u^2\right)e^{-\lambda u^2\left(\frac{3-r}{4}-\frac{\delta}{3}\right)} \frac{du}{\pi} 
\,,\end{aligned}}$$
that is
\begin{equation}\label{eq11}
\left\vert \int_{\alpha-\delta}^{\alpha+\delta}e^{\lambda\left(  f\left(\rho e^{i\theta}\right) - f\left(\rho e^{i\alpha}\right) \right)} \frac{d\theta}{2\pi}
- \int_{\alpha-\delta}^{\alpha+\delta} e^{ -\lambda\frac{(\theta-\alpha)^2}{2}} \frac{d\theta}{2\pi} \right\vert
\le \frac{1}{6\pi\lambda\left(\frac{3-r}{4}-\frac{\delta}{3}\right)^2}+\frac{r-1}{16\sqrt{\pi\lambda}\left(\frac{3-r}{4}-\frac{\delta}{3}\right)^{3/2}}\,.
\end{equation}
We still have
\begin{equation}\label{eq12}
\left\vert  \int_{\vert \theta-\alpha\vert>\delta} e^{ -\lambda\frac{(\theta-\alpha)^2}{2}  }\frac{d\theta}{2\pi} \right\vert
\le \frac{e^{ -\lambda\frac{\delta^2}{2}  }}{\pi\lambda\delta}
\end{equation}
and
\begin{equation}\label{eq14}
\int_{-\infty}^{\infty} e^{ -\lambda\frac{(\theta-\alpha)^2}{2}  }\frac{d\theta}{2\pi} = \frac{1}{\sqrt{2\pi\lambda}}\,.
\end{equation}

Let us now assume $0\le\theta\le\alpha-\delta$. As in the proof of Lemma \ref{key-}, we deduce from \eqref{fg} and Condition 2:
\begin{equation}\label{eq13}
\left\vert \int_0^{\alpha-\delta_1}e^{\lambda\left(  f\left(\rho e^{i\theta}\right) - f\left(\rho e^{i\alpha}\right) \right)} \frac{d\theta}{2\pi} \right\vert
\le \frac{\alpha-\delta_1}{2\pi} e^{-\lambda\frac{\delta_1^2}{2}+\lambda\frac{\delta_1^3}{2}}\,.
\end{equation}
Put $u(\theta)=-\frac{(\theta-\alpha)^2}{2}-\frac{(\theta-\alpha)^3}{2}$, so that $g(\theta)\le u(\theta)$ for $\theta\le\alpha$. We check that
$u'(\theta)=\frac{3}{2}(\alpha-\theta)(\frac{2}{3}-\alpha+\theta)\ge0$ for $\alpha-2/3\le\theta\le \alpha$, and
$u'(\theta)-u'(\alpha-\delta)=-(\theta+\alpha+\delta)(3(\theta-\alpha-\delta)/2+1)\ge0$ for $0\le \theta\le \alpha-\delta$. We thus get $\frac{u'(\theta)}{u'(\alpha-\delta)}\ge 1$ for 
$0\le \theta\le \alpha-\delta$. We deduce
\begin{equation}\label{eq15}
\left\vert \int_{\alpha-\delta_1}^{\alpha-\delta}e^{\lambda\left(  f\left(\rho e^{i\theta}\right) - f\left(\rho e^{i\alpha}\right) \right)} \frac{d\theta}{2\pi} \right\vert
\le \int_{\alpha-\delta_1}^{\alpha-\delta}  \frac{u'(\theta)}{u'(\alpha-\delta)}e^{\lambda u(\theta)}\frac{d\theta}{2\pi}
=\frac{e^{-\lambda\frac{\delta_1^2}{2}+\lambda\frac{\delta_1^3}{2}}-e^{-\lambda\frac{\delta^2}{2}+\lambda\frac{\delta^3}{2}}}{3\pi\delta(2/3-\delta)}\,.
\end{equation}
The estimates \eqref{eq13} and \eqref{eq15} give
\begin{equation}\label{eq16}
\left\vert \int_0^{\alpha-\delta}e^{\lambda\left(  f\left(\rho e^{i\theta}\right) - f\left(\rho e^{i\alpha}\right) \right)} \frac{d\theta}{2\pi} \right\vert
\le \frac{\alpha-\delta_1}{2\pi} e^{-\lambda\frac{\delta_1^2}{2}+\lambda\frac{\delta_1^3}{2}}+\frac{e^{-\lambda\frac{\delta_1^2}{2}+\lambda\frac{\delta_1^3}{2}}}{\pi\delta(2-3\delta)}\,.
\end{equation}

Similarly we have
\begin{equation}\label{eq17}
\left\vert \int_{\alpha+\delta}^{\pi}e^{\lambda\left(  f\left(\rho e^{i\theta}\right) - f\left(\rho e^{i\alpha}\right) \right)} \frac{d\theta}{2\pi} \right\vert
\le \frac{\pi-\alpha-\delta_1}{2\pi} e^{-\lambda\frac{\delta_1^2}{2}+\lambda\frac{\delta_1^3}{2}}+\frac{e^{-\lambda\frac{\delta_1^2}{2}+\lambda\frac{\delta_1^3}{2}}}{\pi\delta(2-3\delta)}\,.
\end{equation}

From \eqref{eq11}, \eqref{eq12},  \eqref{eq14}, \eqref{eq16}, \eqref{eq17} we obtain
$$\displaylines{\left\vert \int_0^{\pi}e^{\lambda\left(  f\left(\rho e^{i\theta}\right) - f\left(\rho e^{i\alpha}\right) \right)} \frac{d\theta}{2\pi}-\frac{1}{\sqrt{2\pi\lambda}}\right\vert
\le  \frac{1}{6\pi\lambda\left(\frac{3-r}{4}-\frac{\delta}{3}\right)^2}+\frac{r-1}{16\sqrt{\pi\lambda}\left(\frac{3-r}{4}-\frac{\delta}{3}\right)^{3/2}}
+\frac{e^{ -\lambda\frac{\delta^2}{2}  }}{\pi\lambda\delta}
\hfill\cr\hfill
+\frac{\pi-2\delta_1}{2\pi} e^{-\lambda\frac{\delta_1^2}{2}+\lambda\frac{\delta_1^3}{2}}+\frac{2e^{-\lambda\frac{\delta_1^2}{2}+\lambda\frac{\delta_1^3}{2}}}{\pi\delta(2-3\delta)}\,,}$$
and the lemma follows from
$$\Re\left( \int_0^{\pi}e^{\lambda\left(  f\left(\rho e^{i\theta}\right) - g(\alpha) \right)} \frac{d\theta}{2\pi}- \frac{e^{i\lambda h(\alpha)}}{\sqrt{2\pi\lambda}}\right)
=\frac{I(\lambda)}{2}-\frac{\cos(\lambda h(\alpha))}{\sqrt{2\pi\lambda}}\,.$$
\end{proof}

\subsection{Estimates for $f$ and $g$}

\begin{lemma}\label{r1f} Assume $1\le  r\le 2.282$. For $\alpha-\sqrt{r-1}\le\theta\le \alpha+\sqrt{r-1}$, we  have
$$\left\vert  f\left( \rho e^{i\theta}\right)- f\left( \rho e^{i\alpha}\right) +\frac{(\theta-\alpha)^2}{2} \right\vert
\le \frac{\vert\theta-\alpha\vert^3}{3}+\frac{r-1}{4} (\theta-\alpha)^2\,.$$
\end{lemma}

\begin{proof} We first notice that $\sqrt{r-1}\le \alpha=\arccos\left(\frac{r-1}{2\sqrt r}\right)$, for $1\le  r\le 2.282459$. This implies $0\le \theta\le2\alpha\le\pi$.

By \eqref{fg} we have
$$\begin{aligned}
g'(\theta)+ih'(\theta)&=ir\frac{(e^{i\theta}-e^{i\alpha})(e^{-i\alpha}-e^{i\theta})}{r-e^{2i\theta}}=ir\frac{\left(e^{i(\theta-\alpha)}-1\right)\left(1-e^{i(\theta+\alpha)}\right)}{r-e^{2i\theta}}\\
&=i\left(e^{i(\theta-\alpha)}-1\right)\left(1+\frac{e^{i(\theta+\alpha)}\left(e^{i(\theta-\alpha)}-r\right)}{r-e^{2i\theta}}\right)\,,
\end{aligned}$$
which gives
$$\left\vert g'(\theta)+ih'(\theta)-i\left(e^{i(\theta-\alpha)}-1\right)\right\vert\le \vert\theta-\alpha\vert \frac{\vert e^{i(\theta-\alpha)}-1\vert+r-1}{\vert r-e^{2i\theta}\vert} 
\le \frac{ (\theta-\alpha)^2 + (r-1)\vert\theta-\alpha\vert}{\vert r-e^{2i\theta}\vert}\,.$$
Since $2\alpha-2\sqrt{r-1}\le 2\pi-2\alpha-2\sqrt{r-1}$, we also get
$$\begin{aligned}
\vert r-e^{2i\theta}\vert^2&=r^2+1-2r\cos(2\theta)\ge r^2+1-2r\cos\left(2\alpha-2\sqrt{r-1}\right)  \\
&\ge r^2+1-2r\cos(2\alpha)- 4r\sqrt{r-1}\sin(2\alpha)+4r(r-1)\cos(2\alpha)\\
&=4r-2(r-1)^{3/2}\sqrt{-r^2+6r-1}+2(r-1)(r^2-4r+1)\\
&= 4+2(r-1)^{3/2}\left(\sqrt{-r^2+6r-1}+(r-3)\sqrt{r-1}\right)\ge4\,,
\end{aligned}$$
which leads to the upper bound
$$\left\vert g'(\theta)+ih'(\theta)-i\left(e^{i(\theta-\alpha)}-1\right) \right\vert\le \frac{(\theta-\alpha)^2}{2}+\frac{r-1}{2}\vert\theta-\alpha\vert\,.$$
We deduce
$$\left\vert g'(\theta)+ih'(\theta)+(\theta-\alpha)\right\vert\le \frac{(\theta-\alpha)^2}{2}+\frac{r-1}{2}\vert\theta-\alpha\vert+\frac{(\theta-\alpha)^2}{2}$$
and the lemma follows by integrating. 
\end{proof}

\begin{lemma}\label{r1g} For any $\alpha/2\le\theta\le3\alpha/2$ and $1<r\le2.11952$, we have
$$g(\theta)-g(\alpha) \le -\frac{(\theta-\alpha)^2}{2}+\frac{\vert\theta-\alpha\vert^3}{2}\,.$$
\end{lemma}

\begin{proof} From Lemma \ref{r-g}  we get
$$g(\theta)-g(\alpha)\le -\frac{(\theta-\alpha)^2}{2}+\frac{\vert\theta-\alpha\vert^3}{2}
+(r-1)(\theta-\alpha)^2\left( \frac{r^2-4r-1}{8r}+\frac{\vert\theta-\alpha\vert}{4}\right)\,.$$
We check that
$$\frac{r^2-4r-1}{16r}+\frac{\alpha}{8}\le0$$
for $1< r\le 2.119518\dots$, and the lemma follows.
\end{proof}

\subsection{Proof of Theorem \ref{thas1}}

For $(r-1)\lambda\ge 702$ and $r-1\le \sqrt{\frac{8\pi}{\lambda}}$, we have $\lambda\ge \lceil \frac{702^2}{8\pi}\rceil=19609$ and therefore $r\le 1+\sqrt\frac{8\pi}{19609}<1.03581$.
We thus can apply Lemmas \ref{r1f} and \ref{r1g}, and we set $\delta=\sqrt{r-1}$ and $\delta_1=\alpha/2$. We then have $\delta\le 0.19<2/3< 3(3-r)/4$.
We therefore can use Lemma \ref{key1}, and we obtain
$$ \displaylines{\left\vert \frac{\sqrt{\pi\lambda}}{2^{\frac{1+(r+1)\lambda}{2}}} I(\lambda)-\cos\left( (r\gamma_1+\gamma_2)\lambda\right) \right\vert
\le \frac{1}{3\sqrt{2\pi\lambda}\left(\frac{3-r}{4}-\frac{\sqrt{r-1}}{3}\right)^2}+\frac{r-1}{2^{7/2}\left(\frac{3-r}{4}-\frac{\sqrt{r-1}}{3}\right)^{3/2}}
\hfill\cr\hfill
+\frac{\sqrt2e^{ -\frac{\lambda(r-1)}{2}  }}{\sqrt{\pi\lambda(r-1)}}
+\left(\frac{\pi-\alpha}{\sqrt{2\pi}} + \frac{2^{3/2}}{\sqrt\pi\delta(2-3\delta)}  \right)\sqrt\lambda e^{-\lambda\frac{\alpha^2}{8}+\lambda\frac{\alpha^3}{16}}\,.}$$
Since $\delta(2-3\delta)\ge \sqrt\frac{702}{\lambda}(2-3\sqrt{0.03581})\ge \frac{37.949}{\sqrt\lambda} $ and $\alpha\le\pi/2$, we find
$$ \displaylines{\left\vert \frac{\sqrt{\pi\lambda}}{2^{\frac{1+(r+1)\lambda}{2}}} I(\lambda)-\cos\left( (r\gamma_1+\gamma_2)\lambda\right) \right\vert
\le \frac{1}{3\sqrt{2\pi\lambda}\left(\frac{3-r}{4}-\frac{\sqrt{r-1}}{3}\right)^2}+\frac{r-1}{2^{7/2}\left(\frac{3-r}{4}-\frac{\sqrt{r-1}}{3}\right)^{3/2}}
\hfill\cr\hfill
+\frac{\sqrt2e^{ -\frac{\lambda(r-1)}{2}  }}{\sqrt{\pi\lambda(r-1)}}
+\left(0.6267+0.04206\sqrt\lambda \right)\sqrt\lambda e^{-\lambda\left(1-\frac{\pi}{4}\right)\frac{\alpha^2}{8}}\,.}$$
Let $\Phi_4(r,\lambda)$ denote this last upper bound. The first two terms and the fourth term are increasing functions of $r\le 1.03581$ and decreasing functions of $\lambda\ge19609$, the third term is a decreasing function
of $\lambda(r-1)\ge702$: we already get this way $\Phi_4(r,\lambda)\le 0.0165$, which proves the first inequality in the theorem.

Further assume $\log\lambda\le (r-1)\lambda\le 1.67007\sqrt{\lambda}$. We find $\lambda\ge \lambda_0:= \lceil \frac{702^2}{1.67007^2}\rceil=176688$, $r\le1+\frac{1.67007}{\lambda_0}$,
$\alpha\ge \alpha_0:=\arccos\left( \frac{\sqrt{1.67007^2/\lambda_0}}{2(1+\sqrt{1.67007^2/\lambda_0})^{1/2}}\right)$ and
$$\displaylines{\sqrt\lambda \Phi_4(r,\lambda)\le \frac{1}{3\sqrt{2\pi}\left(\frac{3-r}{4}-\frac{\sqrt{r-1}}{3}\right)^2}+\frac{1.67007}{2^{7/2}\left(\frac{3-r}{4}-\frac{\sqrt{r-1}}{3}\right)^{3/2}}
\hfill\cr\hfill
+\frac{1}{\sqrt{351\pi}}+\left(0.6267+0.04206\sqrt\lambda_0 \right)\lambda_0 e^{-\lambda_0\left(1-\frac{\pi}{4}\right)\frac{\alpha_0^2}{8}}<1.05882\,,}$$
which proves the second inequality in the theorem for $\log\lambda\le (r-1)\lambda\le 1.67007\sqrt{\lambda}$. 

Assume now $c_1\sqrt\lambda\le  (r-1)\lambda\le c_2 \sqrt\lambda$, so that $\lambda\ge \lambda_0:= \lceil \frac{702^2}{c_2}\rceil$. We find here
$$\displaylines{\sqrt\lambda \Phi_4(r,\lambda)\le \frac{1}{3\sqrt{2\pi}\left(\frac{3-r}{4}-\frac{\sqrt{r-1}}{3}\right)^2}+\frac{c_2}{2^{7/2}\left(\frac{3-r}{4}-\frac{\sqrt{r-1}}{3}\right)^{3/2}}
\hfill\cr\hfill
+\frac{\sqrt2e^{ -\frac{c_1\sqrt\lambda_0}{2}  }}{\sqrt{\pi c_1\sqrt\lambda_0}}
+\left(0.6267+0.04206\sqrt\lambda_0 \right)\lambda_0 e^{-\lambda_0\left(1-\frac{\pi}{4}\right)\frac{\alpha_0^2}{8}}\,,}$$
which gives
$$\sqrt\lambda \Phi_4(r,\lambda)\le\begin{cases}
 1.05882 &\text{if $1.67007\sqrt{\lambda}\le (r-1)\lambda\le \sqrt{\pi\lambda}$;}\\
1.30775 &\text{if $\sqrt{\pi\lambda}\le (r-1)\lambda\le \sqrt{2\pi\lambda}$;}\\
1.50929&\text{if $\sqrt{2\pi\lambda}\le (r-1)\lambda\le \sqrt{3\pi\lambda}$;}\\
1.68876&\text{if $\sqrt{3\pi\lambda}\le (r-1)\lambda\le \sqrt{4\pi\lambda}$;}\\
1.85482&\text{if $\sqrt{4\pi\lambda}\le (r-1)\lambda\le \sqrt{5\pi\lambda}$;}\\
2.01189&\text{if $\sqrt{5\pi\lambda}\le (r-1)\lambda\le \sqrt{6\pi\lambda}$;}\\
2.1626&\text{if $\sqrt{6\pi\lambda}\le (r-1)\lambda\le \sqrt{7\pi\lambda}$;}\\
2.30865&\text{if $\sqrt{7\pi\lambda}\le (r-1)\lambda\le \sqrt{8\pi\lambda}$.}\\
\end{cases}$$
This completes the proof of the theorem.

We now need lower bounds for $\vert \cos(\lambda_1\gamma_1+\lambda_2\gamma_2)\vert$, and this is the aim of the next subsection.

\subsection{Additional lemmas}

\begin{lemma}\label{gamma}
For $1\le r\le \frac{9+\sqrt{73}}{4} =4.386\dots$, we have
$$ 0\le  r\gamma_1+\gamma_2-(r-3)\frac{\pi}{4}+\frac{(r-1)^2}{4} \le \frac{(r-1)^3}{8}\,.$$
\end{lemma}

\begin{proof} Define $\gamma(r)=r\gamma_1+\gamma_2$. We have $\gamma_1(1)=\frac{\pi}{4}$, $\gamma_2(1)=-\frac{3\pi}{4}$, $\gamma_1'(r)=-\frac{1}{r\sqrt{-r^2+6r-1}}$,
$\gamma_2'(r)=\frac{1}{\sqrt{-r^2+6r-1}}$, and therefore  $\gamma(1)=-\frac{\pi}{2}$, $\gamma'=\gamma_1$, $\gamma''(r)=-\frac{1}{r\sqrt{-r^2+6r-1}}$, $\gamma''(1)=-\frac{1}{2}$,
$$\gamma^{(3)}(r)= \frac{-2r^2+9r-1}{r^2(-r^2+6r-1)^{3/2}}\quad\text{and}\quad \gamma^{(4)}(r)=-2\frac{3r^4-27r^3+70r^2-15r+1}{r^3(-r^2+6r-1)^{5/2}}\,.$$
Since $3r^4-27r^3+70r^2-15r+1>0$, the function $\gamma^{(3)}$ is decreasing on $\lbrack1,3+2\sqrt2)$. From $\gamma^{(3)}(1)=3/4$ and $\gamma^{(3)}(\frac{9+\sqrt{73}}{4} )=0$, we get $0\le \gamma^{(3)}(r)\le3/4$ for $1\le r\le \frac{9+\sqrt{73}}{4}$, and
the required inequality follows.
\end{proof}

\begin{lemma}\label{sin}  For $1\le r\le 3$, we have
\begin{itemize}
\item For $\lambda_1+\lambda_2\equiv \modd{0} {4}$:
$$\displaylines{\vert \cos(\lambda(r\gamma_1+\gamma_2))\vert
\hfill\cr\hfill
\ge\begin{cases}
\cos\left(\frac{(r-1)^2}{4}\lambda \right)  &\text{if  $\lambda\frac{(r-1)^2}{4}\le \frac{\pi}{2}$;}\\
\min\left( \cos\left(\pi-\frac{(r-1)^2}{4}\lambda\right), \cos\left(\pi-\frac{3-r}{2}\frac{(r-1)^2}{4}\lambda\right)\right)
&\text{if $\frac{\pi}{3-r}\le \lambda\frac{(r-1)^2}{4}\le\frac{3\pi}{2}$.}
\end{cases}}$$
\item For $\lambda_1+\lambda_2\equiv \modd{1} {4}$:
$$\displaylines{\vert \cos(\lambda(r\gamma_1+\gamma_2))\vert
\hfill\cr\hfill
\ge\begin{cases}
\min\left(\frac{1}{\sqrt2}, \cos\left(\frac{\pi}{4}-\lambda\frac{(r-1)^2}{4} \right)\right) &\text{if  $\lambda\frac{(r-1)^2}{4}\le \frac{3\pi}{4}$;}\\
\min\left( \cos\left(\frac{5\pi}{4}-\frac{(r-1)^2}{4}\lambda\right), \cos\left(\frac{5\pi}{4}-\frac{3-r}{2}\frac{(r-1)^2}{4}\lambda\right)\right)
&\text{if $\frac{3\pi}{2(3-r)}\le \lambda\frac{(r-1)^2}{4}\le\frac{7\pi}{4}$.}
\end{cases}}$$
\item For $\lambda_1+\lambda_2\equiv \modd{2} {4}$:
$$\displaylines{\vert \cos(\lambda(r\gamma_1+\gamma_2))\vert
\hfill\cr\hfill
\ge\begin{cases}
\min\left( \cos\left(\frac{\pi}{2}-\lambda\frac{(r-1)^2}{4} \right),  \cos\left(\frac{\pi}{2}-\frac{3-r}{2}\lambda\frac{(r-1)^2}{4} \right)\right)
&\text{if  $\lambda\frac{(r-1)^2}{4}\le \pi$;}\\
\min\left( \cos\left(\frac{3\pi}{2}-\lambda\frac{(r-1)^2}{4} \right),  \cos\left(\frac{3\pi}{2}-\frac{3-r}{2}\lambda\frac{(r-1)^2}{4} \right)\right)
&\text{if $\frac{2\pi}{3-r}\le \lambda\frac{(r-1)^2}{4}\le2\pi$.}
\end{cases}}$$
\item For $\lambda_1+\lambda_2\equiv \modd{3} {4}$:
$$\displaylines{\vert \cos(\lambda(r\gamma_1+\gamma_2))\vert
\hfill\cr\hfill
\ge\begin{cases}
 \cos\left(\frac{\pi}{4}+\lambda\frac{(r-1)^2}{4} \right) &\text{if  $\lambda\frac{(r-1)^2}{4}\le \frac{\pi}{4}$;}\\
\min\left( \cos\left(\frac{3\pi}{4}-\lambda\frac{(r-1)^2}{4} \right),  \cos\left(\frac{3\pi}{4}-\frac{3-r}{2}\lambda\frac{(r-1)^2}{4} \right)\right)
 &\text{if $\frac{\pi}{2(3-r)}\le \lambda\frac{(r-1)^2}{4}\le\frac{5\pi}{4}$.}
 \end{cases}}$$
\end{itemize}
\end{lemma}

\begin{proof} We deduce from Lemma \ref{gamma} the inequalities
$$0\le \lambda_1\gamma_1+\lambda_2\gamma_2-(\lambda_1-3\lambda_2)\frac{\pi}{4}+\lambda_2\frac{(r-1)^2}{4} \le \lambda_2\frac{(r-1)^3}{8}\,.$$
Let us define $\eta=\lambda_1\gamma_1+\lambda_2\gamma_2-(\lambda_1-3\lambda_2)\frac{\pi}{4} $
so that we  have
$$-\frac{(r-1)^2}{4}\lambda_2 \le \eta\le \frac{r-3}{2}\frac{(r-1)^2}{4}\lambda_2\le 0\,,$$
for $1\le r\le 3$. We then find
$$\vert \cos(\lambda_1\gamma_1+\lambda_2\gamma_2)\vert
=\begin{cases}
\cos\eta  &\text{if $\lambda_1+\lambda_2\equiv \modd{0} {4}$ and $-\frac{\pi}{2}\le \eta\le0$;}\\
\cos\left(\pi+\eta\right)  &\text{if $\lambda_1+\lambda_2\equiv \modd{0} {4}$ and $-\frac{3\pi}{2}\le \eta\le-\frac{\pi}{2}$;}\\
 \cos\left(\frac{\pi}{4}+\eta\right) &\text{if $\lambda_1+\lambda_2\equiv \modd{1} {4}$ and $-\frac{3\pi}{4}\le \eta\le0$;}\\
 \cos\left(\frac{5\pi}{4}+\eta\right) &\text{if $\lambda_1+\lambda_2\equiv \modd{1} {4}$ and $-\frac{7\pi}{4}\le \eta\le-\frac{3\pi}{4}$;}\\
 \cos\left(\frac{\pi}{2}+\eta\right) &\text{if $\lambda_1+\lambda_2\equiv \modd{2} {4}$ and $-\pi\le \eta\le0$;}\\
 \cos\left(\frac{3\pi}{2}+\eta\right) &\text{if $\lambda_1+\lambda_2\equiv \modd{2} {4}$ and  $-2\pi\le \eta\le-\pi$;}\\
 \cos\left(-\frac{\pi}{4}+\eta\right)  &\text{if $\lambda_1+\lambda_2\equiv \modd{3} {4}$ and $-\frac{\pi}{4}\le \eta\le0$;}\\
  \cos\left(\frac{3\pi}{4}+\eta\right)  &\text{if $\lambda_1+\lambda_2\equiv \modd{3} {4}$ and $-\frac{5\pi}{4}\le \eta\le-\frac{\pi}{4}$.}
\end{cases}$$
The lemma follows.
\end{proof}

\subsection{Proof of Theorem \ref{thr1}}

For the sake of contradiction, assume $I(\lambda_2)=0$ with $\lambda_1-\lambda_2\ge702$, by Proposition \ref{propr1}. Theorem \ref{thas1} then gives an upper bound for
$\vert \cos(\lambda_1\gamma_1+\lambda_2\gamma_2)\vert$ that may be smaller than the lower bound given in Lemma \ref{sin}. Because of the form of Theorem \ref{thas1} and Lemma  \ref{sin},
we need to distinguish several cases, according to the residue class of $\lambda_1+\lambda_2$ modulo $4$ and to the size of $\frac{(\lambda_1-\lambda_2)^2}{4\lambda_2}$.
When needed, we shall use the upper bound $r\le 1.03581$ and the estimate $\arccos x\ge \frac{\pi}{2}-\frac{x}{\sqrt{1-x^2}}$ for $0<x<1$.

\subsubsection{The case $\lambda_1+\lambda_2\equiv \modd{0} {4}$}
\begin{itemize}
\item For  $702\le \lambda_1-\lambda_2\le \log\lambda_2$, we obtain the bounds $\lambda_2\ge e^{702}$ and $\frac{(\lambda_1-\lambda_2)^2}{4\lambda_2}\le \frac{702^2}{4e^{702}}<\frac{\pi}{2}$.
We deduce from Theorem \ref{thas1} and Lemma \ref{sin} the contradiction:
$$0.9999\le \cos\left( \frac{702^2}{4e^{702}}\right)\le    \cos\left(\frac{(\lambda_1-\lambda_2)^2}{4\lambda_2}\right)\le  \vert \cos(\lambda(r\gamma_1+\gamma_2))\vert\le 0.0165\,.$$

\item For  $\max(702,\log\lambda_2) \le \lambda_1-\lambda_2\le \sqrt{2\pi\lambda_2}-1.0443$, 
we obtain the lower bound $\lambda_2\ge \left\lceil \frac{703.0443^2}{2\pi}\right\rceil=78666$. We deduce from Theorem \ref{thas1} and Lemma \ref{sin} the inequality
$\cos\left(\frac{(\lambda_1-\lambda_2)^2}{4\lambda_2}\right) \le \frac{1.30775}{\sqrt\lambda_2} $, and we get the contradiction
$$\frac{(\lambda_1-\lambda_2)^2}{4\lambda_2}\ge \arccos \frac{1.30775}{\sqrt\lambda_2}\ge \frac{\pi}{2}- \frac{1.30775}{\sqrt{\lambda_2-1.30775^2}}
> \left(\sqrt\frac{\pi}{2}- \frac{0.52215}{\sqrt\lambda_2} \right)^2\,.$$ 

\item For $\max(702,\sqrt{2\pi\lambda_2}+3.1407)\le  \lambda_1-\lambda_2\le \sqrt{3\pi\lambda_2}$, we obtain the lower bound $\lambda_2\ge \left\lceil \frac{702^2}{3\pi}\right\rceil=52289$.
We deduce from Theorem \ref{thas1} and Lemma \ref{sin} the inequality $\cos\left(\pi-\frac{3-r}{2}\frac{(r-1)^2}{4}\lambda_2\right)\le \frac{1.50929}{\sqrt\lambda_2} $, and we get the contradiction
$$\frac{\pi}{2}- \frac{1.50929}{\sqrt{\lambda_2-1.50929^2}} \le  \pi-\frac{3-r}{2}\frac{(r-1)^2}{4}\lambda_2
\le  \pi -\frac{\pi}{2}\left(1-\frac{\sqrt{3\pi}}{2\sqrt{\lambda_2}} \right)\left(1+\frac{3.1407}{\sqrt{2\pi\lambda_2}}  \right)^2\,.$$

\item For $\max(702,\sqrt{3\pi\lambda_2})\le  \lambda_1-\lambda_2\le \sqrt{\frac{16}{5-r}\pi\lambda_2}<\sqrt{12.68\lambda_2}$, we obtain the lower bound  $\lambda_2\ge \left\lceil \frac{702^2}{12.68}\right\rceil=38865$.
 We deduce from Theorem \ref{thas1} and Lemma \ref{sin} the inequalities $\cos\left(\pi-\frac{3-r}{2}\frac{(r-1)^2}{4}\lambda_2\right)\le \frac{1.85482}{\sqrt\lambda_2} <0.00941$, and we get the contradiction
 $$1.5613<\arccos(0.00941)<\pi-\frac{3-r}{2}\frac{(r-1)^2}{4}\lambda_2\le \pi-\frac{3-r}{8}3\pi\le 0.8276\,.$$
 
 \item For $\max\left(702, \sqrt{\frac{16}{5-r}\pi\lambda_2}\right)\le \lambda_1-\lambda_2\le\sqrt{6\pi\lambda_2}-0.9275$, we obtain the lower bound  $\lambda_2\ge \left\lceil \frac{702.9275^2}{6\pi}\right\rceil=26214$.
 We deduce from Theorem \ref{thas1} and Lemma \ref{sin} the inequality $\cos\left(\pi-\frac{(r-1)^2}{4}\lambda_2\right)\le \frac{2.01189}{\sqrt\lambda_2}$, and we get the contradiction
 $$\frac{\pi}{2}- \frac{2.01189}{\sqrt{\lambda_2-2.01189^2}} \le \lambda_2\frac{(r-1)^2}{4}-\pi\le
\frac{3\pi}{2}\left(1-\frac{0.9275}{\sqrt{6\pi\lambda_2}} \right)^2-\pi\,.$$
 \end{itemize}

  \subsubsection{The case $\lambda_1+\lambda_2\equiv \modd{1} {4}$}
  
  \begin{itemize}
  
\item  For  $703\le \lambda_1-\lambda_2\le \sqrt{2\pi\lambda_2}$, we obtain the bound $\lambda_2\ge \left\lceil \frac{703^2}{2\pi}\right\rceil=78656$.
We deduce from Theorem \ref{thas1} and Lemma \ref{sin} the contradiction:
$$\frac{1}{\sqrt2}\le \vert \cos(\lambda(r\gamma_1+\gamma_2))\vert\le\min\left(0.0165,\frac{1.30775}{\sqrt{\lambda_2}}\right)\,.$$

\item For  $\max(703,\sqrt{2\pi\lambda_2})\le \lambda_1-\lambda_2\le \sqrt{3\pi\lambda_2}-0.984$, we   obtain the bound $\lambda_2\ge \left\lceil \frac{703.984^2}{3\pi}\right\rceil=52585$.
The inequality $\cos\left(\frac{\pi}{4}-\frac{(r-1)^2}{4}\lambda_2 \right)\le \frac{1.50929}{\sqrt\lambda_2} $ follows from Theorem \ref{thas1} and Lemma \ref{sin}, and we get the contradiction
$$\frac{(\lambda_1-\lambda_2)^2}{4\lambda_2}\ge \frac{\pi}{4}+\arccos\left( \frac{1.50929}{\sqrt\lambda_2}\right)\ge \frac{3\pi}{4}- \frac{1.50929}{\sqrt{\lambda_2-1.50929^2}}
>\left(\sqrt\frac{3\pi}{4}- \frac{0.492}{\sqrt\lambda_2} \right)^2\,.$$

\item For $\max(703,\sqrt{3\pi\lambda_2}+3.8433)\le  \lambda_1-\lambda_2\le \sqrt{4\pi\lambda_2}$, we obtain the lower bound $\lambda_2\ge \left\lceil \frac{702^2}{4\pi}\right\rceil=39217$.
The inequality $\cos\left(\frac{5\pi}{4}-\frac{3-r}{2}\frac{(r-1)^2}{4}\lambda_2\right)\le \frac{1.68876}{\sqrt\lambda_2} $ follows from Theorem \ref{thas1} and Lemma \ref{sin} , and we get the contradiction
$$\frac{\pi}{2}- \frac{1.68876}{\sqrt{\lambda_2-1.68876^2}} \le  \frac{5\pi}{4}-\frac{3-r}{2}\frac{(r-1)^2}{4}\lambda_2
\le \frac{5\pi}{4} -\frac{3\pi}{4}\left(1-\frac{\sqrt{\pi}}{\sqrt{\lambda_2}} \right)\left(1+\frac{3.8433}{\sqrt{3\pi\lambda_2}}  \right)^2\,.$$

\item For $\max(703,\sqrt{4\pi\lambda_2})\le  \lambda_1-\lambda_2\le \sqrt{\frac{20}{5-r}\pi\lambda_2}<\sqrt{15.85\lambda_2}$, we obtain the lower bound  $\lambda_2\ge \left\lceil \frac{703^2}{15.85}\right\rceil=31181$.
 We deduce from Theorem \ref{thas1} and Lemma \ref{sin} the inequalities $\cos\left(\frac{5\pi}{4}-\frac{3-r}{2}\frac{(r-1)^2}{4}\lambda_2\right)\le \frac{2.01189}{\sqrt\lambda_2} <0.0114$, and we get the contradiction
 $$1.5593<\arccos(0.0114)<\frac{5\pi}{4}-\frac{3-r}{2}\frac{(r-1)^2}{4}\lambda_2\le \frac{5\pi}{4}-\frac{3-r}{8}4\pi< 0.8417\,.$$

\item  For $\max\left(703, \sqrt{\frac{20}{5-r}\pi\lambda_2}\right)\le \lambda_1-\lambda_2\le\sqrt{7\pi\lambda_2}-0.9231$, we obtain the lower bound  $\lambda_2\ge \left\lceil \frac{703.9231^2}{7\pi}\right\rceil=22533$.
We deduce from Theorem \ref{thas1} and Lemma \ref{sin} the inequality $\cos\left(\frac{5\pi}{4}-\frac{(r-1)^2}{4}\lambda_2\right)\le \frac{2.1626}{\sqrt\lambda_2}$, and we get the contradiction
$$\frac{\pi}{2}- \frac{2.1626}{\sqrt{\lambda_2-2.1626^2}} \le \frac{(r-1)^2}{4}\lambda_2-\frac{5\pi}{4}\le \frac{7\pi}{4}\left(1-\frac{0.9231}{\sqrt{7\pi\lambda_2}} \right)^2-\frac{5\pi}{4}\,.$$

\end{itemize}
  
  \subsubsection{The case $\lambda_1+\lambda_2\equiv \modd{2} {4}$}

  \begin{itemize}
  
  \item For $\max\left(702,2.0582\lambda_2^{1/4}\right)\le \lambda_1-\lambda_2\le \sqrt{\pi\lambda_2}$, we obtain the lower bound  $\lambda_2\ge \left\lceil \frac{702^2}{\pi}\right\rceil=156865$.
  The inequality $\cos\left(\frac{\pi}{2}-\frac{3-r}{2}\frac{(r-1)^2}{4}\lambda_2\right)\le \frac{1.05582}{\sqrt\lambda_2}$ follows from Theorem \ref{thas1} and Lemma \ref{sin} , and we get the contradiction
 $$ \frac{\pi}{2}- \frac{1.05882}{\sqrt{\lambda_2-1.05882^2}} \le   \frac{\pi}{2}-\frac{3-r}{2}\frac{(r-1)^2}{4}\lambda_2 \le  \frac{\pi}{2}- \left(1-\frac{2.0582}{2\lambda_2^{3/4}}\right)\frac{2.0582^2}{4\sqrt\lambda_2}\,.$$
  
  \item For $\max\left(702, \sqrt{\pi\lambda_2}\right)\le \lambda_1-\lambda_2\le \sqrt{\frac{8}{5-r}\pi\lambda_2}<\sqrt{6.34\lambda_2}$, we obtain the lower bound  $\lambda_2\ge \left\lceil \frac{702^2}{6.34}\right\rceil=77730$.
 We deduce from Theorem \ref{thas1} and Lemma \ref{sin} the inequalities $\cos\left(\frac{\pi}{2}-\frac{3-r}{2}\frac{(r-1)^2}{4}\lambda_2\right)\le \frac{1.50929}{\sqrt\lambda_2}<0.00542$, and we get the contradiction
 $$1.5653<\arccos(0.00542)<\frac{\pi}{2}-\frac{3-r}{2}\frac{(r-1)^2}{4}\lambda_2\le \frac{\pi}{2}-\frac{3-r}{8}\pi< 0.7995\,.$$
  
  \item For $\max\left(702,\sqrt{\frac{8}{5-r}\pi\lambda_2}\right) \le \lambda_1-\lambda_2\le  2\sqrt{\pi\lambda_2}-0.9535$, we obtain the lower bound  $\lambda_2\ge \left\lceil \frac{702.9535^2}{4\pi}\right\rceil=39323$.
 We deduce from Theorem \ref{thas1} and Lemma \ref{sin} the inequality $\cos\left(\frac{\pi}{2}-\frac{(r-1)^2}{4}\lambda_2\right)\le \frac{1.68876}{\sqrt\lambda_2}$, and we get the contradiction
 $$\frac{\pi}{2}- \frac{1.68876}{\sqrt{\lambda_2-1.68876^2}} \le \frac{(r-1)^2}{4}\lambda_2-\frac{\pi}{2}\le \pi\left(1-\frac{0.9535}{2\sqrt{\pi\lambda_2}} \right)^2-\frac{\pi}{2}\,.$$
  
  \item  For $\max\left(702,2\sqrt{\pi\lambda_2}+4.5938\right) \le \lambda_1-\lambda_2\le \sqrt{5\pi\lambda_2}$, we obtain the lower bound  $\lambda_2\ge \left\lceil \frac{702^2}{5\pi}\right\rceil=31373$.
We deduce from Theorem \ref{thas1} and Lemma \ref{sin} the inequality $\cos\left(\frac{3\pi}{2}-\frac{3-r}{2}\frac{(r-1)^2}{4}\lambda_2\right)\le \frac{1.85482}{\sqrt\lambda_2} $, and we get the contradiction
$$\frac{\pi}{2}- \frac{1.85482}{\sqrt{\lambda_2-1.85482^2}} \le  \frac{3\pi}{2}-\frac{3-r}{2}\frac{(r-1)^2}{4}\lambda_2
\le  \frac{3\pi}{2} -\pi\left(1-\frac{\sqrt{5\pi}}{2\sqrt{\lambda_2}} \right)\left(1+\frac{4.5938}{2\sqrt{\pi\lambda_2}}  \right)^2\,.$$

 \item For $\sqrt{5\pi\lambda_2}  \le \lambda_1-\lambda_2\le \sqrt{\frac{24}{5-r}\pi\lambda_2}<\sqrt{19.02\lambda_2}$, we obtain the lower bound  $\lambda_2\ge \left\lceil \frac{702^2}{19.02}\right\rceil=25910$.
 The inequalities $\cos\left(\frac{3\pi}{2}-\frac{3-r}{2}\frac{(r-1)^2}{4}\lambda_2\right)\le \frac{2.1626}{\sqrt\lambda_2}<0.01344$ follow from Theorem \ref{thas1} and Lemma \ref{sin}, and we get the contradiction
$$1.5573<\arccos(0.01344)<\frac{3\pi}{2}-\frac{3-r}{2}\frac{(r-1)^2}{4}\lambda_2\le \frac{3\pi}{2}-\frac{3-r}{8}5\pi< 0.8558\,.$$

\item For $\max\left(702,  \sqrt{\frac{24}{5-r}\pi\lambda_2}\right) \le \lambda_1-\lambda_2\le 2\sqrt{2\pi\lambda_2}-0.9218$, we obtain the lower bound  $\lambda_2\ge \left\lceil \frac{702.9218^2}{8\pi}\right\rceil=19660$.
We deduce from Theorem \ref{thas1} and Lemma \ref{sin} the inequality $\cos\left(\frac{3\pi}{2}-\frac{(r-1)^2}{4}\lambda_2\right)\le \frac{2.30865}{\sqrt\lambda_2}$, and we get the contradiction
$$\frac{\pi}{2}- \frac{2.30865}{\sqrt{\lambda_2-2.30865^2}} \le \frac{(r-1)^2}{4}\lambda_2-\frac{3\pi}{2}\le 2\pi\left(1-\frac{0.9218}{2\sqrt{2\pi\lambda_2}} \right)^2-\frac{3\pi}{2}\,.$$

  \end{itemize}
  
  \subsubsection{The case $\lambda_1+\lambda_2\equiv \modd{3} {4}$}
  
  \begin{itemize}
  
  \item For  $703\le \lambda_1-\lambda_2\le \log\lambda_2$, we still find $ \frac{(\lambda_1-\lambda_2)^2}{4\lambda_2}\le \frac{703^2}{4e^{703}}<\frac{\pi}{2}$ and we get the contradiction 
  $ \cos\left(\frac{\pi}{4}+\frac{703^2}{4e^{703}}\right)  \le 0.0165$.
  
  \item For  $\max(703,\log\lambda_2)\le \lambda_1-\lambda_2\le \sqrt{\pi\lambda_2}-1.1958$, we obtain the bound $\lambda_2\ge \left\lceil \frac{704.1958^2}{\pi}\right\rceil=157848$.
  The inequality $\cos\left(\frac{\pi}{4}+\frac{(r-1)^2}{4}\lambda_2 \right)\le \frac{1.05882}{\sqrt\lambda_2} $ follows from Theorem \ref{thas1} and Lemma \ref{sin}, and we get the contradiction
  $$\frac{(\lambda_1-\lambda_2)^2}{4\lambda_2}\ge \arccos\left( \frac{1.05882}{\sqrt\lambda_2}\right)-\frac{\pi}{4}\ge \frac{\pi}{4}- \frac{1.05882}{\sqrt{\lambda_2-1.05882^2}}
>\left(\sqrt\frac{\pi}{4}-\frac{0.5979}{\sqrt\lambda_2}\right)^2\,.$$

\item For $\max(703,\sqrt{\pi\lambda_2}+2.5913) \le \lambda_1-\lambda_2\le \sqrt{2\pi\lambda_2}$, we obtain the bound $\lambda_2\ge \left\lceil \frac{703^2}{2\pi}\right\rceil=78656$.
The inequality $\cos\left(\frac{3\pi}{4}-\frac{3-r}{2}\frac{(r-1)^2}{4}\lambda_2\right)\le \frac{1.30775}{\sqrt\lambda_2} $ follows from Theorem \ref{thas1} and Lemma \ref{sin} , and we get the contradiction
$$\frac{\pi}{2}- \frac{1.30775}{\sqrt{\lambda_2-1.30775^2}} \le  \frac{3\pi}{4}-\frac{3-r}{2}\frac{(r-1)^2}{4}\lambda_2
\le \frac{3\pi}{4} -\frac{\pi}{4}\left(1-\frac{\sqrt{\pi}}{\sqrt{2\lambda_2}} \right)\left(1+\frac{2.5913}{\sqrt{\pi\lambda_2}}  \right)^2\,.$$

\item For $\max(703,\sqrt{2\pi\lambda_2})\le  \lambda_1-\lambda_2\le \sqrt{\frac{12}{5-r}\pi\lambda_2}<\sqrt{9.51\lambda_2}$, we obtain the lower bound  $\lambda_2\ge \left\lceil \frac{703^2}{9.51}\right\rceil=51968$.
 We deduce from Theorem \ref{thas1} and Lemma \ref{sin} the inequalities $\cos\left(\frac{3\pi}{4}-\frac{3-r}{2}\frac{(r-1)^2}{4}\lambda_2\right)\le \frac{1.68876}{\sqrt\lambda_2} <0.00741$, and we get the contradiction
 $$1.5633<\arccos(0.00741)<\frac{3\pi}{4}-\frac{3-r}{2}\frac{(r-1)^2}{4}\lambda_2\le \frac{3\pi}{4}-\frac{3-r}{8}2\pi< 0.8136\,.$$
 
 \item For $\max\left( 703,\sqrt{\frac{12}{5-r}\pi\lambda_2}\right) \le  \lambda_1-\lambda_2\le \sqrt{5\pi\lambda_2}-0.9367$, we obtain the bound $\lambda_2\ge \left\lceil \frac{703.9367^2}{5\pi}\right\rceil=31547$.
 We deduce from Theorem \ref{thas1} and Lemma \ref{sin} the inequality $\cos\left(\frac{3\pi}{4}-\frac{(r-1)^2}{4}\lambda_2\right)\le \frac{1.85482}{\sqrt\lambda_2}$, and we get the contradiction
$$\frac{\pi}{2}- \frac{1.85482}{\sqrt{\lambda_2-1.85482^2}} \le \frac{(r-1)^2}{4}\lambda_2-\frac{3\pi}{4}\le \frac{5\pi}{4}\left(1-\frac{0.9367}{\sqrt{5\pi\lambda_2}} \right)^2-\frac{3\pi}{4}\,.$$
 
   \end{itemize}

\section{Concluding remarks}

In the introduction, we discussed the irreducibility of $C_{X,\lambda_2}$, and noticed how different are the cases $\lambda_2$ even and $\lambda_2$ odd.
We checked both cases for $\lambda_2\le 240$, but we can go further in the even case. Using {\tt Maple} during 35101 seconds, we showed that  $C_{X,\lambda_2}$
is irreducible over $\Q$ when $\lambda_2\le 600$ is even. This motivates the following conjecture.

\begin{conjecture} For $\lambda_2\ge2$ even, the polynomial $C_{X,\lambda_2}$ is irreducible over $\Q$. For $\lambda_2\ge3$ odd, the polynomial $C_{X,\lambda_2}$ is the product of
$X-\lambda_2$ by an irreducible polynomial over $\Q$. 
\end{conjecture}

We used the same technics, together with hypergeometric transformations, to study the case $\lambda_1-\lambda_2$ small and to prove Proposition \ref{propr1}.
We introduced the four families of polynomials and checked their irreducibility over $\Q$ for $l\le 350$. It is quite likely that this property holds for any larger value of $l$.

\begin{conjecture} For $l\ge3$, the polynomials $\widetilde C_{l,0,0}(X)$, $\widetilde C_{l,0,1}(X)$, $\widetilde C_{l,1,0}(X)$ and $\widetilde C_{l,1,1}(X)$ are irreducible over $\Q$.
\end{conjecture}
  
 Let us now discuss the results obtained in Theorem \ref{thr1}. The first thing we noticed is that the case $\lambda_1+\lambda_2\equiv \modd{2} {4}$ differs from the other cases. It would be nice to be nice to deal with the interval $702\le \lambda_1-\lambda_2\le 2.0582\lambda_2^{1/4}$ for any $\lambda_2$, to fill the initial gap. Secondly we chose to reach the second explicit intervals with no solutions. How far could we go with this method? It would be nice to get improvements that enable to break the $\sqrt\lambda_2$ barrier and to go up to $\lambda_2^{1/2+\epsilon}$ for some $\epsilon>0$. 
 
 \section{Acknowledgments}
 
 The author wishes to thank the anonymous referee for substantially helping to improve the presentation of this paper in many ways.

\bigskip
\hrule
\bigskip

\noindent 2010 {\it Mathematics Subject Classification}:
Primary 11B65, Secondary 05A10 11B83.

\noindent \emph{Keywords: }
binomial sums, asymptotics

\bigskip
\hrule
\bigskip

\end{document}